\documentclass{amsart}

\usepackage{floatrow}

\usepackage{blindtext}
\usepackage{amsfonts,amssymb,amsmath,amsthm, multicol, caption, subcaption}
\usepackage{tikz}
\usepackage{tikz-cd}
\usepackage{tkz-euclide}
\usepackage{hyperref}
\usepackage[all,cmtip]{xy}
\usepackage[utf8]{inputenc}
\usepackage[nameinlink]{cleveref}

\newcommand{\kk}{\mathbb{K}}
\newcommand{\m}{\mathbf{m}}

\newcommand{\Q}{\mathcal{Q}}
\newcommand{\cD}{\mathcal{D}}

\newcommand{\cR}{\mathcal{R}}

\newcommand{\Z}{\mathbb{Z}}
\newcommand{\A}{\mathcal{A}}
\newcommand{\B}{\mathcal{B}}

\DeclareMathOperator{\syz}{syz}
\DeclareMathOperator{\reg}{reg}
\DeclareMathOperator{\cl}{cl}
\DeclareMathOperator{\rank}{rank}
\DeclareMathOperator{\codim}{codim}
\DeclareMathOperator{\im}{Im}

\newcommand{\p}{\mathbf{p}}

\newcommand{\q}{\mathbf{q}}

\newcommand{\bv}{\mathbf{v}}

\newcommand{\dx}{\frac{\partial}{\partial x}}
\newcommand{\dy}{\frac{\partial}{\partial y}}
\newcommand{\dz}{\frac{\partial}{\partial z}}
\newcommand{\bbP}{\mathbb{P}}
\newcommand{\bp}{\mathbf{p}}
\newcommand{\bq}{\mathbf{q}}
\newcommand{\sat}{\text{\footnotesize sat}}

\newcommand{\matroid}{\mathcal{M}}

\newcommand{\wrep}{\text{weak P-Rep}}

\usepackage[normalem]{ulem}

\newtheorem{thm}{Theorem}[section]
\newtheorem{cor}[thm]{Corollary}
\newtheorem{lem}[thm]{Lemma}
\newtheorem{prop}[thm]{Proposition}

\theoremstyle{definition}
\newtheorem{defn}[thm]{Definition}
\newtheorem{exm}[thm]{Example}
\newtheorem{remark}[thm]{Remark}

\newtheorem{ques}[thm]{Question}
\newtheorem{notation}[thm]{Notation}

\newfloatcommand{capbtabbox}{table}[][\FBwidth]

\usepackage{hyperref}

\title{Geometric aspects of the Jacobian of a hyperplane arrangement}
\author{Michael DiPasquale}
\address{Department of Mathematics and Statistics, University of South Alabama}
\email{mdipasquale@southalabama.edu}
\author{Jessica Sidman}
\address{Department of Mathematics and Statistics, Amherst College}
\email{jsidman@amherst.edu}
\author{Will Traves}
\address{Department of Mathematics, United States Naval Academy}
\email{traves@usna.edu}

  \thanks{
;  Secondary: 
52C25 
}

\begin{document}

\maketitle

\begin{abstract}
An embedding of the complete bipartite graph $K_{3,3}$ in $\mathbb{P}^2$ gives rise to both a line arrangement and a bar-and-joint framework. For a generic placement of the six vertices,  the graded Betti numbers of the logarithmic module of derivations of the line arrangement are constant, but an example due to Ziegler shows that the graded Betti numbers are different when the points lie on a conic. 

 
 Similarly, in rigidity theory a generic embedding of $K_{3,3}$ in the plane is an infinitesimally rigid bar-and-joint framework, but the framework is infinitesimally flexible when the points lie on a conic. In this paper we develop the theory of weak perspective representations of hyperplane arrangements to formalize and generalize the striking connection between hyperplane arrangements and rigidity theory that this example suggests.  In particular, we seek to understand how the interplay of combinatorics and geometry influence  algebraic structures associated to an arrangement, such as the saturation of the Jacobian ideal of the arrangement. 
 We make connections between examples and constructions from rigidity theory and interesting phenomena in the study of hyperplane arrangements.


\end{abstract}

\section{Introduction}\label{sec:Intro}

Let $\A \subset \kk^{\ell+1}$ be a central hyperplane arrangement  over a field $\kk$ of characteristic zero 
-- all hyperplanes pass through the origin -- and let $D(\A)$ be the module of logarithmic derivations, consisting of polynomial vector fields tangent to $\A$.  Inspired by examples from rigidity theory, a main aim of this paper is to explore how the combinatorics and geometry of $\A$ influence the saturation of the Jacobian ideal of $\A$ and, ultimately, the graded Betti numbers of $D(\A)$.  We introduce and develop the notion of a weak perspective representation of the matroid of a hyperplane arrangement, inspired by Whiteley's notion of {\em parallel drawings} of an incidence structure~\cite[Chapter~61]{DiscCompGeom18} (see also~\cite{Whiteley87,Whiteley97,NixonSchulzeWhiteley22}). We draw an explicit connection between the geometric notion of a weak perspective representation of $\A$ and the saturation of its Jacobian ideal.

We use the rigidity theory literature to present a number of examples of line arrangements $\A$  (including several infinite families) where $D(\A)$  possesses a minimal free resolution sensitive to the geometry of the line arrangement.   For instance, in Example~\ref{ex: glueing K33} we use a gluing construction from rigidity theory to construct a line arrangement from $m$ copies of $K_{3,3}$ with the property that by deforming the geometry (without changing the intersection lattice) the number of minimal first syzygies of $D(\A)$ of degree $8m$ 
can take any integer value between 0 and $m$. 
Several other similar examples are explored in Section~\ref{sec:ExtremalSyzygies}.  We hope that this paper fosters closer connections between researchers in the communities studying hyperplane arrangements and rigidity theory.

Our inspiration comes from Ziegler's pair \cite{Schenck-2012}, which consists of two line arrangements, $\A$ and $\A'$, both obtained by extending the edges of an embedded complete bipartite graph $K_{3,3}$ in $\mathbb{P}^2$ with intersection lattice consisting of six triple points and 18 double points. In the arrangement $\A$ the six triple points are generic, but in $\A'$ they lie on a conic, and $D(\A')$ has different graded Betti numbers than $D(\A)$ (Ziegler's original condition is stated in terms of certain linear dependencies -- Schenck points out in~\cite{Schenck-2012} that the linear conditions of Ziegler amount to the six triple points lying on a conic).

The geometry underlying Ziegler's pair is identical to that of a well-known example in rigidity theory.  A generic embedding of $K_{3,3}$ in the plane as a bar-and-joint framework is infinitesimally rigid, but if the six joints lie on a conic then it is infinitesimally flexible.  (See \cite{wunderlich} for an exposition.)  As Whiteley observes in~\cite{Whiteley89}, by an `old engineering technique,' a non-trivial infinitesimal motion of the framework is equivalent to a non-trivial parallel redrawing of the bars. Such a non-trivial parallel redrawing is shown in Figure~\ref{fig:K33nondegConic}.  
\begin{figure}[h!]
    \centering
    \includegraphics[width = 4cm]{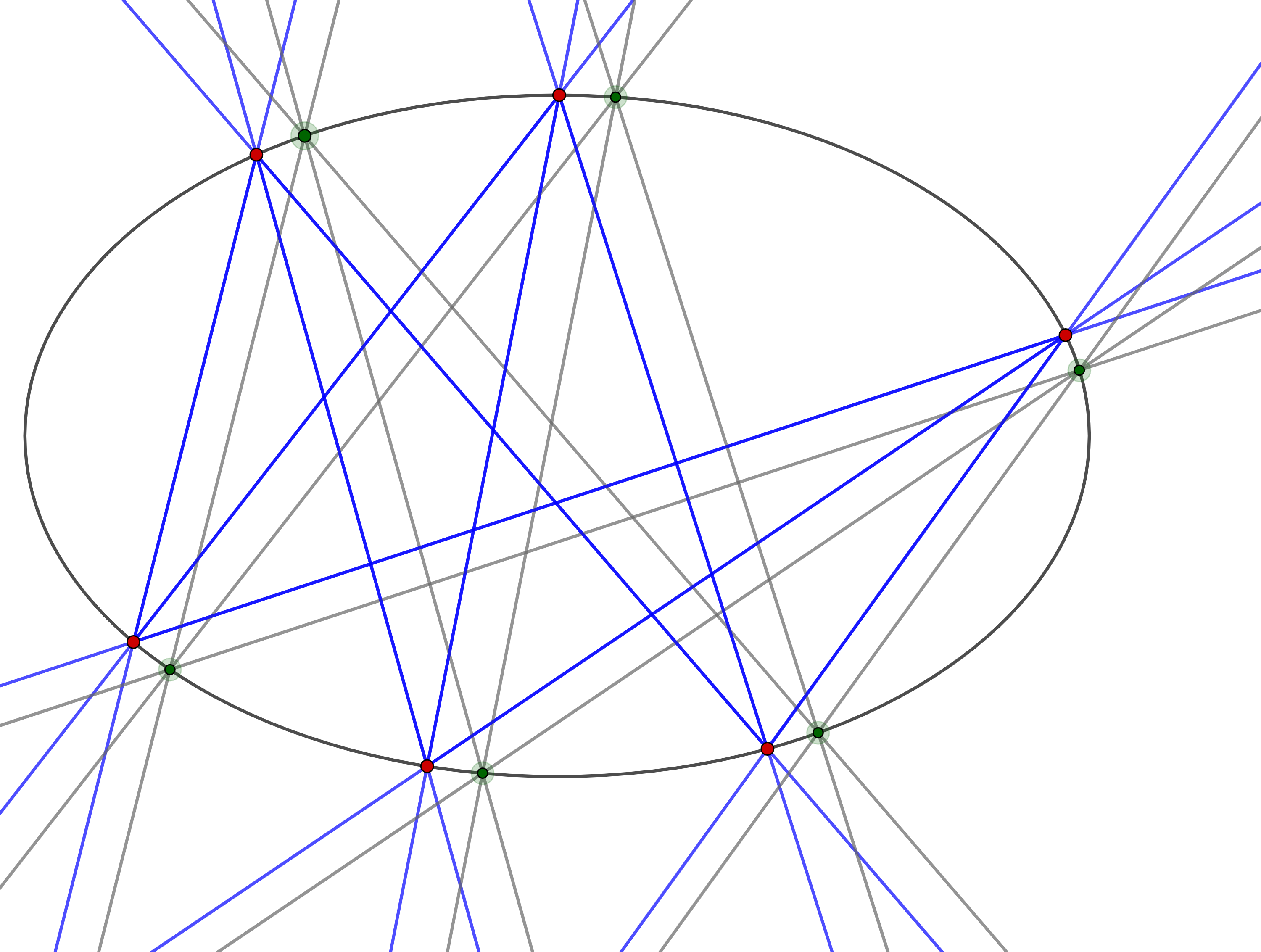}
    \caption{When the vertices of $K_{3,3}$ lie on a conic, the arrangement has non-trivial parallel redrawings, equivalently non-trivial weak perspective representations.}
    \label{fig:K33nondegConic}
\end{figure}
We make this correspondence precise in Section~\ref{sec:PerspectiveMatroidRepresentations}, where the notion of a weak perspective representation generalizes the notion of a parallel redrawing.

 Ziegler offers his pair of line arrangements as an indication of the subtlety of a conjecture that has become known as \textit{Terao's conjecture} on the freeness of hyperplane arrangements.  An arrangement 
$\A$ is called \textit{free} if its module of logarithmic derivations $D(\A)$ is a free module over the polynomial ring.  Terao's conjecture~\cite[Conjecture~4.138]{OT92} proposes that freeness over a fixed field is combinatorial (that is, it depends only on the lattice $L(\A)$). 
In Ziegler's pair, neither line arrangement is free, but the example shows that the algebraic structure of the module of logarithmic derivations depends in subtle ways on the geometry of the arrangement.  Indeed, understanding how the algebraic structure of $D(\A)$ interacts with the geometry and combinatorics of $\A$ is at the heart of Terao's conjecture.  While none of the examples from this paper yield a counterexample, they underscore that rigidity theory provides a good place to seek a counterexample to Terao's conjecture.


Our motivation to connect rigidity to hyperplane arrangements stems from the notion of formality, introduced by Falk and Randell~\cite{FalkRandellHomotopy87}.  Falk and Randell call an arrangement \textit{formal} if all relations among its defining linear forms are generated by relations of length three.  This interesting property is also studied in connection with the Orlik-Terao algebra~\cite{Schenck-Tohaneanu-2009,Le-Mohammadi-2015}, which is a commutative analogue of the celebrated Orlik-Solomon algebra.

Falk and Randell conjectured, and Yuzvinsky proved~\cite{Yuz93A}, that free arrangements are formal. They also asked whether formality is combinatorial but  Yuzvinsky~\cite[Example~2.2]{Yuz93A} showed that Ziegler's pair is a counterexample. In the same paper, Yuzvinsky asked about the relationship between the graded Betti numbers of $D(\A)$ and the formality of $\A$.  As we discuss below, we provide a precise answer to this question for line arrangements.

One of our contributions in this paper is a series of tight connections between the geometric notions of formality, perspective representations, and rigidity.  
We characterize the formality of an arrangement in terms of its persepective representations. In Section~\ref{sec:PerspectiveMatroidRepresentations} we show an arrangement is formal if and only if it admits only trivial perspective representations.  
In particular, 
we show that a line arrangement is formal if and only if an associated graph is rigid (Proposition~\ref{prop:correspondences}).  It is well-known in the rigidity theory literature that rigidity of a graph is a geometric (not combinatorial) property, and hence so is formality.

Another contribution of this paper is a complete description, in characteristic zero, of polynomials of minimal degree in various saturations of the Jacobian ideal of a hyperplane arrangement $\A$ in terms of perspective representations of $\A$ (Theorem~\ref{thm:SaturationParallelDrawing}).  As a corollary, if an arrangement fails to be formal then its module of logarithmic derivations cannot be free.  This recovers (in characteristic zero) the result of Yuzvinsky~\cite{Yuz93A}.  In particular, for line arrangements, we specify in Corollary~\ref{cor:DerivationSyzygies}
exactly how freeness fails when the arrangement is not formal -- the failure of formality contributes a specific graded Betti number to the minimal free resolution of $D(\A)$.

 We provide a road map to the paper.  We begin in Section~\ref{sec:FormalityInPlane} by reviewing the notion of \textit{formal} and $k$-generated arrangements, following a presentation of Brandt and Terao~\cite{Brandt-Terao-1994} and Tohaneanu~\cite{Tohaneanu-K-Formal-07}.
 While nothing in this section is new, the underlying homological perspective underpins our approach.

In Section~\ref{sec:PerspectiveMatroidRepresentations} we introduce \textit{weak perspective representations of hyperplane arrangements up to rank }$k$. As we detail in Corollary~\ref{cor:FormalityAndParallelDrawings}, a hyperplane arrangement is formal
if and only if it has only trivial weak perspective representations up to rank $3$. 
This point of view is particularly rich for line arrangements.  In Section~\ref{sec:RigidityPlanarFrameworks} we recall some basics from planar rigidity theory. 
Proposition~\ref{prop:correspondences} describes the correspondence between 
non-trivial weak perspective representations of a line arrangement and non-trivial infinitesimal motions of a framework.  A wealth of interesting examples flow from this connection, which we begin to explore in Section~\ref{sec:ExtremalSyzygies}.

In Section~\ref{sec:ModuleDerivations} we recall the relationship between the module of logarithmic derivations of an arrangement and its Jacobian ideal in characteristic zero.  We introduce the $k$-saturation of the Jacobian ideal.  Following this, we show in Section~\ref{sec:SaturatingJacobian} that there is an explicit connection between weak perspective representations of a hyperplane arrangement up to rank $k+1$ and polynomials of minimal degree in the $k$-saturation of the Jacobian ideal.  This result, Theorem~\ref{thm:SaturationParallelDrawing}, is the main tool that we use to connect the algebra of the module of logarithmic derivations to weak perspective representations of a hyperplane arrangement.  As a corollary, we obtain a generalization (in characteristic zero) of Yuzvinsky's result that free arrangements are formal~\cite{Yuz93A}.  Namely, in Corollary~\ref{cor:MaxPDim} we show that if a hyperplane arrangement $\A$ admits non-trivial weak perspective representations up to rank $k+1$, then its module of logarithmic derivations has projective dimension at least $k-1$.

In Section~\ref{sec:ExtremalSyzygies} we specialize to the case of line arrangements, where we have our strongest results.  The results are primarily due to a duality for almost complete intersections of dimension one which we first saw in a paper of Hassanzadeh and Simis~\cite{Hassanzadeh-Simis-2012}, but which appeared in earlier work of Chardin~\cite{Chardin-2004}. Similar constructions were also considered by Pellikaan~\cite{Pellikaan-1988}.
This duality has been studied in several additional papers 
\cite{Sernesi-2014,Straten-Warmt-2015,Failla-Flores-Peterson-2021}. Applying this duality to the Jacobian ideal of a line arrangement yields a connection, detailed in Corollary~\ref{cor:DerivationSyzygies}, between polynomials of minimal degree in the saturation of the Jacobian of a line arrangement and syzygies of maximal degree in the module of derivations.  The further correspondence between weak perspective representations of line arrangements and infinitesimal rigidity of embedded graphs from Proposition~\ref{prop:correspondences} yields many line arrangements from rigidity theory whose module of logarithmic derivations exhibits a geometric sensitivity.  In Corollary~\ref{cor:DerivationRegularity} we observe that formal line arrangements with $n$ lines are characterized by possessing a module of logarithmic derivations with (Castelnuovo-Mumford) regularity strictly less than $n-2$.  As shown by Schenck~\cite{SchenckElementaryModifications03}, the maximum regularity of the module of logarithmic derivations of a line arrangement with $n$ lines is $n-2$.  Thus formal line arrangements are precisely those which fail to have maximum regularity.

We end the paper by indicating additional connections to the literature and posing a number of questions.  We highlight two of these.  

Question~\ref{ques:regularity} concerns extending the characterization of formality of line arrangements via regularity to higher dimensions.  A result of Derksen and Sidman~\cite{Derksen-Sidman-2004}, improved recently by Saito~\cite{saito2019degeneration}, shows that the regularity of the module of logarithmic derivations of a hyperplane arrangement with $n$ hyperplanes in $\bbP^\ell$ is at most $n-\ell$.
This bound is sharp for generic arrangements.  
We do not know if there is any connection between maximal regularity and formality in higher dimensions though this may be an interesting direction for future research.

We ask in Question~\ref{ques:BernsteinSato} if there is an explicit relation between formality of an arrangement and its \textit{Bernstein-Sato polynomial}. 
 In~\cite[Section~5.3]{Walther-2017}, Walther proves that the Bernstein-Sato polynomial of the two line arrangements $\A$ (six triple points not on a conic) and $\A'$ (six triple points on a conic) in Ziegler's pair are different.  Walther shows that the difference hinges on the existence of a polynomial of maximum possible degree in the saturation of the Jacobian of $\A'$. 
 In Theorem~\ref{thm:SaturationParallelDrawing} we show that this difference in the saturation is predicted by formality.
 It would be fascinating if there is a deeper connection between the Bernstein-Sato polynomial and the related ideas of formality, weak perspective representations, and rigidity for line arrangements.

\section{Formality of hyperplane arrangements}\label{sec:FormalityInPlane}
Let $V\cong\kk^{\ell+1}$, $\bbP^\ell=\bbP(V)$ be the projectivization of $V$, and $V^*$ be the dual vector space.  Define the map $\alpha:\{1,\ldots,n\}\rightarrow V^*$, and let $\alpha_i = \alpha(i)$.  The hyperplane arrangement defined by $\alpha$ is the union $\A=\cup_{i=1}^n H_i\subset \bbP^\ell$ of $n$ distinct hyperplanes in $\bbP^\ell$ with $H_i=\mathbb{V}(\alpha_i)$, where $\mathbb{V}(\alpha_i)$ denotes the zero locus of the linear form $\alpha_i$.  We write $\A(\alpha)$ when we wish to emphasize the choice of linear forms defining the hyperplanes.
In this section we show that $\A$ is formal if and only if the space of relations on $\alpha_1,\ldots,\alpha_n$ is generated by local relations, a condition that can be checked via homology.  This is well-known (see~\cite{Tohaneanu-K-Formal-07,Brandt-Terao-1994}), but serves as a useful introduction to our matroidal and homological perspective.

Let $M(\alpha) = [\alpha_1| \cdots | \alpha_n]$ be the $(\ell+1)\times n$ matrix whose columns are the coefficients of $\alpha_1,\ldots,\alpha_n$. Define the \textit{relation space} of $\A$ to be the kernel of $M(\alpha)$, 
\[
\cR(\A) = \{\bv \in \mathbb{K}^n : \; M(\alpha)\bv = \mathbf{0} \},
\]
and the \textit{length} of a relation $\bv\in\cR(\A)$ to be the number of non-zero entries in $\bv$.

\begin{defn}[\cite{FalkRandellHomotopy87}]\label{defn:formal}
A hyperplane arrangement is \textit{formal} if its relation space is generated by relations of length three.
\end{defn}

We define $\matroid(\alpha)$ to be the linear matroid on the columns of $M(\alpha)$.  This matroid contains the combinatorial data of the arrangement. We briefly recall the definitions from matroid theory \cite{Oxley} that we will use in what follows.  If $\mathcal{N}$ is a matroid with ground set $E$, and $X \subseteq E$, the \emph{rank} of $X$ is the cardinality of the largest independent set contained in $X$.  The \emph{closure} of $X$ is the maximal set $Y$ containing $X$ with $\rank Y = \rank X$.  A set is closed if it is equal to its closure, and the closed subsets of $E$ are the \emph{flats} of the matroid.  The \textit{lattice of flats} of $\mathcal{N}$ consists of the flats of $\mathcal{N}$, ordered with respect to reverse inclusion.  If $X$ is a flat of $\mathcal{N}$, then we write $n_{X}$ for the cardinality of $X$.

If $\A(\alpha)$ is a hyperplane arrangement, then we can give a geometric interpretation of the matroid definitions.  In particular, the flats of the matroid $\matroid(\alpha)$ are also flats in the intersection lattice of the arrangement $\A(\alpha)$.  If $X = \{\alpha_{i_1}, \ldots, \alpha_{i_{n_X}}\}$ is a flat of $\matroid(\alpha)$, then $\mathbb{V}(\alpha_{i_1}, \ldots, \alpha_{i_{n_X}}) \subset \bbP^{\ell}$ is a flat of $\A(\alpha)$, and we denote both flats by $X$ so that $\rank X = \codim X.$  If $\rank(X)=\ell$, then $X$ is a point in $\mathbb{P}^\ell$ and we write $p$ instead of $X$.  We denote the lattice of flats of $\matroid(\alpha)$ by $L(\A)$, and write $L_k(\A)$ for the flats of rank $k$. 
The hyperplane arrangment $\A$ is called \textit{essential} if $\rank(\A)=\ell+1$.  That is, $\cap_{i=1}^n H_i=\emptyset\subset\bbP^\ell$, which we will say has codimension $\ell+1$.  If $\A$ is not essential, we will use the convention that if $\rank(\A)\le k\le \ell+1$, then $L_{k}(\A)=L_{\rank(\A)}(\A)$. 

Let $X$ be a flat of $\matroid(\alpha)$.  We write $\alpha_X$ for the restriction of $\alpha$ to $\{j : \alpha_j \in X\}$ and we call $\A_X = \A(\alpha_X)$ the \textit{localization} of $\A$ at $X$.
We define $\cR_X(\A) \subseteq \cR(\A)$ to be the subspace of linear relations on the $n_X$ forms in $X$.  $\cR_X(\A)$ is isomorphic to $\cR(\A_X)$ via the natural inclusion of $\cR(\A_X)$ into $\cR(\A)$. Theorem \ref{thm:local} characterizes formal arrangements as those where the `local' relations around codimension two flats $X$ with $n_X \geq 3$ generate the entire relation space of $\A$ (see~\cite[Remark~3.5]{Brandt-Terao-1994}).

\begin{thm}\label{thm:local}
An arrangement $\A$ is formal if and only if 
$\displaystyle{\sum_{X\in L_2(\A)} \cR_X(\A)=\cR(\A)}$.
\end{thm}
\begin{proof}
Suppose that $\A$ is formal and $\bv \in \mathbb{K}^n$ is a relation of length three on $\alpha_i, \alpha_j,$ and $\alpha_k$.  Then $X=\mathbb{V}(\alpha_i,\alpha_j,\alpha_k)$ has codimension two and $\bv \in \cR_X(\A)$.  Since $\A$ is formal, $\cR(\A)$ is generated by relations of length three and we conclude that $\cR(\A)  \subseteq \sum_{X\in L_2(\A)} \cR_X(\A).$

Now assume that $\sum_{X\in L_2(\A)} \cR_X(\A)=\cR(\A)$.  It suffices to show that $\cR_X(\A)$ is generated by relations of length three whenever $X\in L_2(\A)$.  Let $M(\alpha_X)$ be the $(\ell+1) \times n_X$ matrix whose columns correspond to $\alpha_X$.  Since $X$ has rank two, the reduced row echelon form of $M(\alpha_X)$ has two pivot columns -- without loss of generality suppose these are the first two columns.  Then a basis for $\cR_X(\A)$, which is the kernel of $M(\alpha_X)$, is given by the relations among $\alpha_1,\alpha_2,$ and $\alpha_j$ for $j=3,\ldots,n_X$.  Thus $\A$ is formal.
\end{proof}

As Tohaneanu noticed in~\cite{Tohaneanu-K-Formal-07}, we can re-cast formality as a homological condition as follows.  For each $X\in L_2(\A)$
choose a basis for $\cR_X(\A)$ consisting of relations of length three, and define $N_X$ to be the $n \times (n_X-2)$ matrix with these vectors as columns.

\begin{cor}\label{cor:HomologicalCriterionFormality}
The arrangement $\A$ is formal if and only if the kernel of $M(\alpha)$ coincides with the image of $\oplus N_X$, which is to say that the chain complex
\begin{equation}\label{eq:formalComplex}
\bigoplus_{X\in L_2(\A)}\kk^{n_X-2}\xrightarrow{\oplus N_X} \kk^n\xrightarrow{M(\alpha)}\kk^{\ell+1}
\end{equation}
is exact.
\end{cor}
\begin{proof}
Note that $\sum \cR_X(\A)$ can be identified with the column space of the $n\times \sum_X (n_X-2)$ matrix $\oplus N_X$ formed by concatenating the matrices $N_X$ according to an arbitrary ordering on the flats $X\in L_2(\A)$.  We thus see that $\sum\cR_X(\A)$ is the image of the linear map
\[
\oplus N_X:\bigoplus_X\kk^{n_X-2}\rightarrow \kk^n
\]
while $\cR(\A)$ is the kernel of the map given by the coefficient matrix
\[
M(\alpha):\kk^n\to \kk^{\ell+1}.
\]
Now the result follows immediately from Theorem \ref{thm:local}.
\end{proof}

\begin{remark}\label{rem:kgenerated}
There is a straightforward generalization of formality which we will use in Section~\ref{sec:PerspectiveMatroidRepresentations}.
In~\cite{Yuz93A}, Yuzvinsky calls an arrangement \textit{$k$-generated} (for $k\ge 3$) if $\cR(\A)$ is generated by relations of length $k$ (hence formal arrangements are $3$-generated).  Theorem~\ref{thm:local} generalizes to yield that $\A$ is $k$-generated if and only if $\sum_{X\in L_{k-1}(\A)} \cR_X(\A)=\cR(\A)$.  Likewise, we may reformulate Corollary~\ref{cor:HomologicalCriterionFormality} to give a homological criterion for an arrangement to be $k$-generated.  We need only let the summation run over the flats $X\in L_{k-1}(\A)$ and the matrix $N_X$ becomes an $n\times (n_X-\rank(X))$ matrix whose columns record the relations of length $\rank(X)+1$ among the $n_X$ forms defining $X$.
\end{remark}

We apply the ideas of this section to show that the line arrangement $D_3$ is formal. 
\begin{exm}\label{ex:formal2}
The $D_3$ line arrangement 
is defined as the zero locus of the linear forms
\[
x-y,\; x+y,\; x-z,\; x+z,\; y-z,\; y+z.
\]
The four triple points of this arrangement are listed in Figure~\ref{fig:TriplePoint}.

\begin{figure}[h!t]
\begin{floatrow}
{
\begin{tikzpicture}[scale=1]
\tkzDefPoints{0/0/A, 2/0/B, 2/2/C, 0/2/D}
\tkzDrawLine(A,B);
\tkzDrawLine(B,C);
\tkzDrawLine(C,D);
\tkzDrawLine(A,D);
\tkzDrawLine(A,C);
\tkzDrawLine(B,D);

\end{tikzpicture}
}
\hspace{10 pt}
\capbtabbox{%
\raisebox{50 pt}{
\begin{tabular}{l|l}
Triple point & \parbox{80 pt}{Forms vanishing on triple point}\\
\hline
$p_1=[1:1:1]$ & $x-y,\; x-z,\; y-z$\\

$p_2=[1:-1:1]$ & $x+y,\; x-z,\; y+z$\\

$p_3=[-1:1:1]$ & $x+y,\; x+z,\; y-z$\\

$p_4=[-1:-1:1]$ & $x-y,\; x+z,\; y+z$ \\
\end{tabular}
}
}{
}
\end{floatrow}
\caption{The line arrangement $D_3$ and its triple points.}
\label{fig:TriplePoint}
\end{figure}
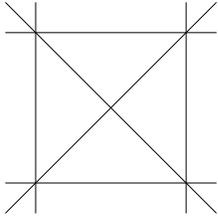

The relation spaces of $\A$ at $p_1,p_2,p_3,$ and $p_4$ are all one-dimensional, and the matrix for $\oplus_{i=1}^4 N_{p_i}$ is
\[
\bigoplus\limits_{i=1}^4 N_{p_i}=
\left[ \begin{array}{rrrr}
1 & 0 & 0 & 1\\
0 & 1 & 1 & 0\\
-1&-1 & 0 & 0\\
0 & 0 &-1 &-1\\
1 & 0 &-1 & 0\\
0 &-1 & 0 & 1 
\end{array} \right]
\]

Hence the chain complex~\eqref{eq:formalComplex} takes the form
\[
\kk^4\xrightarrow{ 
\left[ \begin{array}{rrrr}
1 & 0 & 0 & 1\\
0 & 1 & 1 & 0\\
-1&-1 & 0 & 0\\
0 & 0 &-1 &-1\\
1 & 0 &-1 & 0\\
0 &-1 & 0 & 1 
\end{array} \right]} \kk^6 \xrightarrow{	\left[ \begin{array}{rrrrrr}
		1 & 1 & 1 & 1 & 0 & 0\\
	   -1 & 1 & 0 & 0 & 1 & 1\\
		0 & 0 &-1 & 1 &-1 & 1
	\end{array} \right]} \kk^3.
\]
The exactness of this chain complex at the middle term can be easily checked by rank computation -- both matrices are rank 3 --  so $\A$ is formal.
\end{exm}

The  homological condition gives an algorithmic way to check formality on a computer algebra system such as Macaulay2~\cite{M2}.  

One may well ask if the homology of the chain complex~\eqref{eq:formalComplex} at the middle term has a geometric meaning.  In the next section we show that this homology encodes \textit{weak} \textit{perspective} representations of the matroid $\matroid(\alpha)$.

\section{$\wrep{}$s of matroids}\label{sec:PerspectiveMatroidRepresentations}
In Section~\ref{sec:FormalityInPlane} we saw that formal and $k$-generated arrangements could be described by the vanishing of the homology at the middle term in a chain complex of length three.  In this section we attach geometric meaning to this homology (more precisely, its dual).  We convey the intuition before diving into the details.

Suppose $\A=\A(\alpha)=\cup_{i=1}^n H_i\subset\bbP^\ell$ and let $\mathbf{\lambda}=(\lambda_i)_{i=1}^n\in \kk^{n}$ be a vector satisfying all the dependencies of length $k$ that the linear forms $\alpha_1,\ldots,\alpha_n$ satisfy.  That is, if $\mathbf{v}\in \cR_X(\A)$ (defined just prior to Theorem~\ref{thm:local}) for some $X\in L_{k-1}(\A)$, then $\mathbf{v}\cdot \mathbf{\lambda}=0$, where $\cdot$ is the usual dot product.  Choose a linear form $\alpha_0$ (we make restrictions on this linear form later).  We can use $\alpha_0$ and the vector $\lambda$ to produce a new collection of linear forms $\beta_i=\alpha_i+\lambda_i\alpha_0$ and thus another hyperplane arrangement $\A'(\beta)$.  Notice that if $H_0=\mathbb{V}(\alpha_0)$, then $\A$ and $\A'$ have the same \textit{restriction} to $H_0$ (that is, the intersections of the hyperplanes of $\A$ and $\A'$ with $H_0$ agree).  In this case, $\A$ and $\A'$ are \textit{perspective from} $H_0$ as in Definition~\ref{defn:perspective}.

With $\beta$ defined as in the previous paragraph, how is the matroid $\matroid(\beta)$ related to $\matroid(\alpha)$?  What does it mean, in terms of $\matroid(\beta)$, if $\lambda$ is in $\im M(\alpha)^T$?  
As long as $\alpha_0$ is sufficiently general, $\beta$ is what we will call a
\textit{weak representation of} $\matroid(\alpha)$ \textit{up to rank $k$}.  If $\lambda$ is in
$\im M(\alpha)^T$ then we will say that $\beta$ is a \textit{trivial} weak representation of $\matroid(\alpha)$.  In the remainder of the section we begin by developing these notions, making the connection with $k$-generated arrangements at the end.  Our development is inspired
by Whiteley's parallel drawings of $d$-scenes in~\cite[Chapter~61]{DiscCompGeom18} (see also~\cite{Whiteley87,Whiteley97,NixonSchulzeWhiteley22}), and our goal is to make connections to other matroids in discrete geometry where there are established tools for studying `special positions' that are relevant for Terao's conjecture.  For instance, in Section~\ref{sec:RigidityPlanarFrameworks} we show that $\wrep{}$s of line arrangements are (very) closely tied to infinitesimal motions of planar frameworks.  Special positions for infinitesimal motions are highly studied in rigidity theory~\cite{WhiteWhiteleyAlgGeoStresses}.
 
 Recall that if $\matroid$ and $\mathcal{N}$ are matroids on the same ground set $\{1,\ldots,n\}$, then $\mathcal{N} \geq \matroid$ in the weak order if every dependent set in $\mathcal{N}$ is also dependent in $\matroid$ (or equivalently, every independent set in $\mathcal{M}$ is independent in $\mathcal{N}$).  
For any $1\le k\le \rank(\matroid)$, the \textit{rank} $k$ \textit{truncation} of a matroid $\matroid$, denoted $\matroid^{[k]}$, is the matroid (of rank $k$) on the same ground set as $\matroid$ whose independent sets are precisely the independent sets of $\matroid$ that have rank at most $k$.  If $\matroid^{[k]}=\mathcal{N}^{[k]}$, then $\matroid$ and $\mathcal{N}$ have the same flats of rank \textit{at most } $k-1$, but
may have different flats of rank $k$ and higher.  See Example~\ref{ex:matroid reps}, where two arrangements have the same rank $2$ truncation but different flats of rank $2$.

\begin{defn}\label{def:weak rep}
Let $\mathcal{N}$ be a matroid on the ground set $\{1,\ldots,n\}$. 
The map $\alpha:\{1,\ldots,n\}\to V^*$ is a \textit{representation} of $\mathcal{N}$ \textit{up to rank } $k$ if $\mathcal{N}^{[k]}=\matroid(\alpha)^{[k]}$ and a \textit{weak representation} of $\mathcal{N}$ \textit{up to rank } $k$ if $\mathcal{N}^{[k]}\ge \matroid(\alpha)^{[k]}$.  We call the hyperplane arrangement $\A(\alpha)$ a (weak) representation of $\mathcal{N}$ up to rank $k$ if $\alpha$ is a (weak) representation of $\mathcal{N}$ up to rank $k$.  If $k = \rank \mathcal{N}$, then we say that $\alpha$ is a (weak) representation of $\mathcal{N}$, accordingly.
\end{defn}

\begin{exm}\label{ex:matroid reps}
Let $\mathcal{N}$ be the matroid on $\{1,2,3\}$ in which every subset is independent.  Define an arrangement $\A$ by the map $\alpha$ where $\alpha_1 = x, \alpha_2=y,$ and $\alpha_3=x+y+z$. We see that $\matroid(\alpha) = \mathcal{N}$, and $\alpha$ is a representation of $\mathcal{N}$.

Let $\A'$ be defined by $\beta$ with $\beta_1=x, \beta_2 =y,$ and $\beta_3=x+y$. The three lines of $\A'$ pass through a common point, which corresponds to  the union of three  rank 2 flats in $\mathcal{N}$.  In this case, $\beta$ is a weak representation of $\mathcal{N}$.  The map $\beta$ is also a representation of $\mathcal{N}$ up to rank 2 as $\mathcal{N}^{[2]} = \matroid(\beta)^{[2]}$.   Both of these arrangements are depicted in Figure~\ref{fig: threeLines}.

\begin{figure}[ht]
    \centering
    \includegraphics[width =8cm]{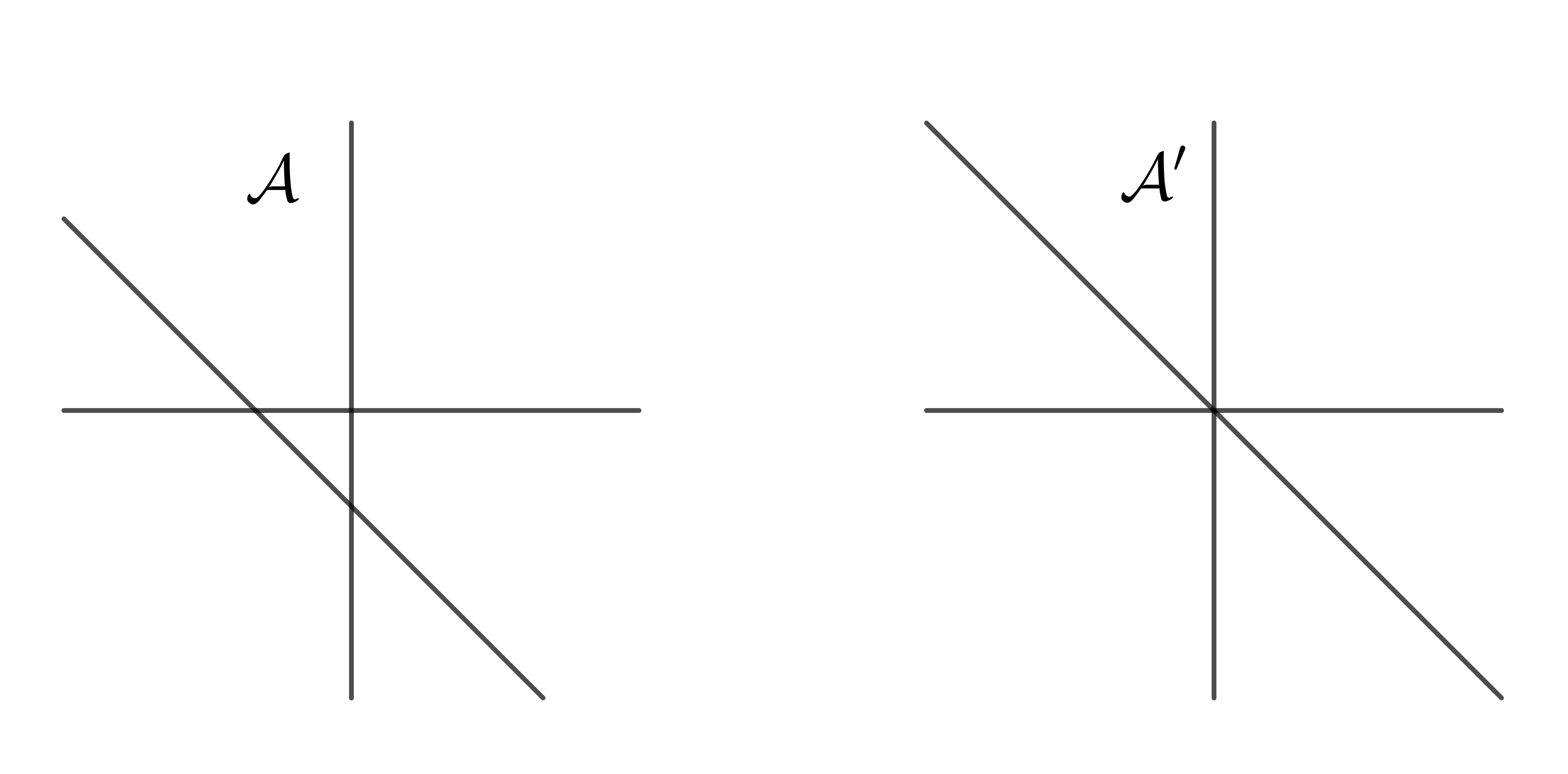}
    \caption{Arrangements giving a representation and a weak representation of a matroid on three elements in which every subset is independent.}
    \label{fig: threeLines}
\end{figure}

\end{exm}

Since the flats of $\matroid(\alpha)$ correspond to the hyperplanes of $\A(\alpha)$ and their intersections, if $\beta$ is a weak representation of $\matroid(\alpha)$, then the correspondence $\alpha_i \mapsto \beta_i$ induces a surjective map on the flats of the matroids.  This map need not be injective, as Example~\ref{ex:matroid reps} shows.

\begin{lem}\label{thm: flat map}
If $\beta$ is a weak representation of
$\matroid(\alpha)$ up to rank $k\geq 3$, then the correspondence $\alpha_i \mapsto \beta_i$ induces a surjective map from the flats of $\matroid(\alpha)^{[k]}$ to the flats of $\matroid(\beta)^{[k]}$ that takes flats of rank 2 to flats of rank 2.  In general, the rank of a flat can drop under this map, but not increase.
\end{lem}

\begin{remark}\label{rem:kge3}
Hereafter we will usually make the restriction that whenever $\beta:\{1,\ldots,n\}\to V^*$ is a weak representation of $\matroid(\alpha)$ up to rank $k$, then $k\ge 3$.  The reason for this is partly due to our use of Lemma~\ref{thm: flat map}.  However, if $\beta$ is only a weak representation of $\matroid(\alpha)$ up to rank $2$, then the only restriction on $\beta$ is that it is an injective map and so represents a hyperplane arrangement of $n$ hyperplanes.  Such representations typically do not preserve enough information to be useful.  
\end{remark}

\begin{proof}[Proof of Lemma~\ref{thm: flat map}]
The correspondence $\alpha_i \mapsto \beta_i$ sends the flat $X = \{\alpha_{i_1}, \ldots, \alpha_{i_t}\}$ to the flat $\cl(\beta_{i_1}, \ldots, \beta_{i_t})$. To see that this map is surjective, let $Y$ be a flat of $\matroid(\beta)^{[k]}$ of rank $r$.  Then there exist independent $\beta_{i_1}, \ldots, \beta_{i_r} \in Y$, and since every independent subset of $\matroid(\beta)^{[k]}$ is independent in $\matroid(\alpha)^{[k]}$, $X = \cl(\alpha_{i_1}, \ldots, \alpha_{i_r})$ is a flat of rank $r$. We will show that if $\alpha_j \in X$, then $\beta_j \in Y$.

If $\alpha_j \in \{\alpha_{i_1}, \ldots, \alpha_{i_r}\}$, we are done, so assume that $\alpha_j \notin \{\alpha_{i_1}, \ldots, \alpha_{i_r}\}$. Since $\{\alpha_{i_1}, \ldots, \alpha_{i_r}, \alpha_j\}$ has rank $r$ it must be dependent.  Since every dependent set in
$\matroid(\alpha)^{[k]}$ is dependent in $\matroid(\beta)^{[k]}$, $\{\beta_{i_1}, \ldots, \beta_{i_r}, \beta_j\}$ is dependent, which implies that $\beta_j \in \cl(\beta_{i_1}, \ldots, \beta_{i_r}) = Y$.

Note that if $\rank(X) = 2,$ our assumption that $\beta$ defines a hyperplane arrangement (so that the $\mathbb{V}({\beta_i})$ are distinct) 
and that $k \geq 3$ 
implies that two independent elements of $X$ map to independent elements in $\matroid(\beta)^{[k]}$.  The final statement that the rank of a flat cannot increase under this map follows directly from the fact that $\beta$ is a weak representation of $\matroid(\alpha)$ up to rank $k$.
\end{proof}

Given a hyperplane arrangement $\A=\cup_{i=1}^n H_i\subset\bbP^\ell$ and a hyperplane $H_0\notin\A$, the \textit{restriction} of $\A$ to $H_0$ is the hyperplane arrangement, in $H_0\cong \bbP^{\ell-1}$, defined as
\[
\A|_{H_0}=\cup_{i=1}^n (H_i\cap H_0).
\]
\begin{defn}\label{defn:perspective}
   Given two hyperplane arrangements $\A(\alpha)$ and $\A'(\beta)$ so that $H_0=\mathbb{V}(\alpha_0)$ is in neither arrangement, we say that $\A$ and $\A'$ are \textit{perspective from} $H_0$ if $\A|_{H_0}=\A'|_{H_0}$.  We also say that the maps $\alpha$ and $\beta$ are perspective from $\alpha_0$. 
\end{defn}

\begin{remark}\label{rem:Parallel}
Our intuition for perspective hyperplane arrangements comes from the situation where $H_0$ is the hyperplane at infinity, in which case the hyperplanes in arrangements that are perspective from $H_0$ are parallel.  Although the theory of parallel drawings of $d$-scenes in \cite[Chapter~61]{DiscCompGeom18} guided much of our thinking, we found that the combinatorial structure given by the matroid of an arrangement was a more natural language for our needs than that of $d$-scenes.

We use the word \textit{perspective} instead of \textit{parallel} hyperplane arrangements because $H_0$ may not always be the `hyperplane at infinity,' and in this case \textit{perspective} matches the terminology used in classical theorems of projective geometry.
\end{remark}

\begin{exm}\label{exm: standard scene}
The line arrangements in Figure~\ref{fig: threeLines} from Example~\ref{ex:matroid reps} are perspective from the line at infinity.

\end{exm}

\begin{exm}\label{exm: desargues}
Figure~\ref{fig:desargues} depicts two arrangements (one solid and the other dashed) that are representations of the matroid $\mathcal{N}$ from Example~\ref{ex:matroid reps} and perspective from the dotted line $H_0$. Since $H_0$ is not at infinity, the corresponding lines in these arrangements are not parallel.

\begin{figure}[ht]
    \centering
    \includegraphics[width =8cm]{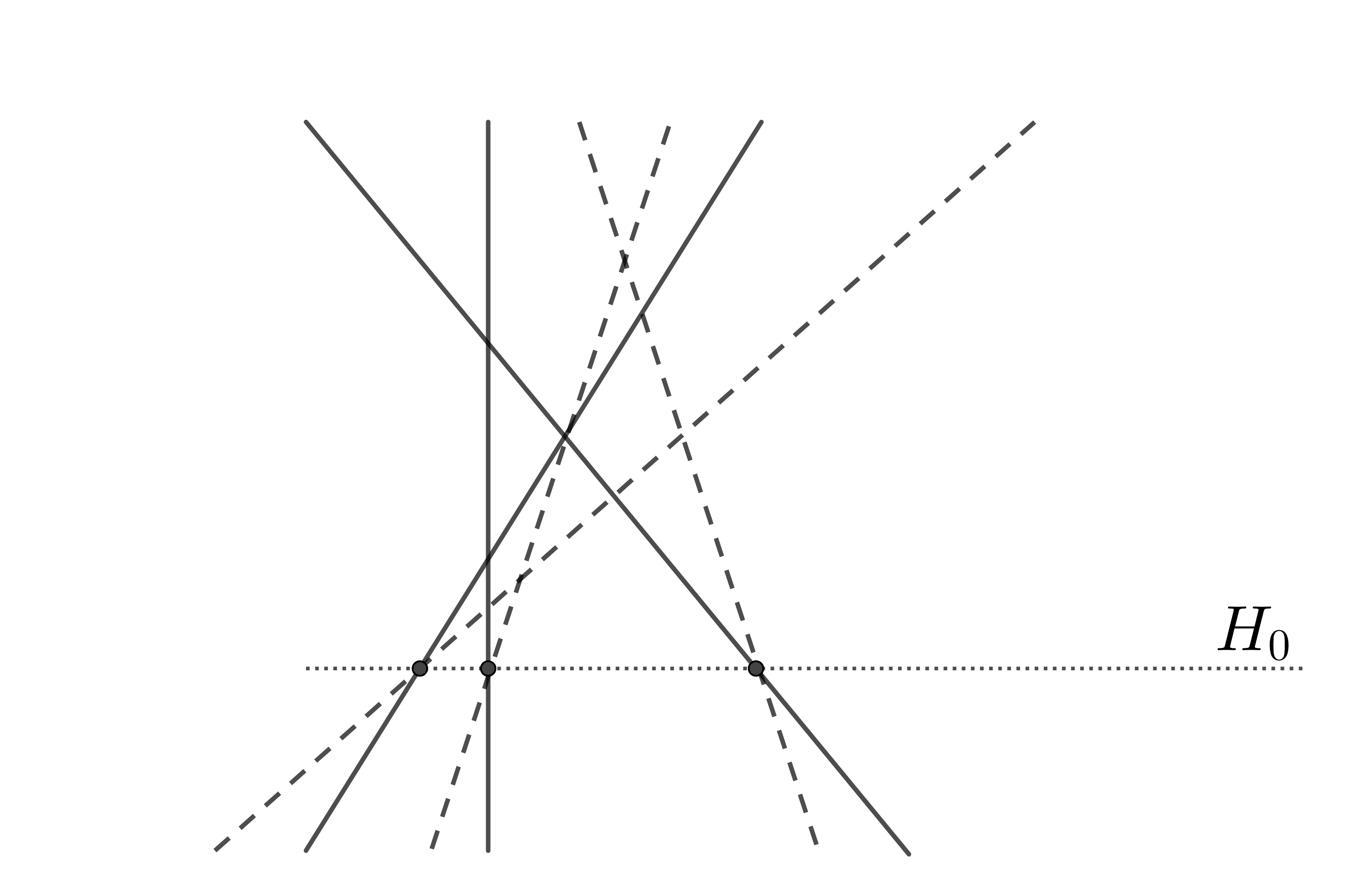}
    \caption{Two arrangements perspective from a general $H_0$.}
    \label{fig:desargues}
\end{figure}

We can show that if we perturb one of the line arrangements in Figure~\ref{fig:desargues}, then the two line arrangements will not be perspective from any line.
Observe that the three solid and three dashed lines in Figure~\ref{fig:desargues} each enclose a triangle.  Desargues' Theorem from projective geometry states that two triangles are perspective from a line if and only if they are perspective from a point.  That is, if we take the three lines through corresponding vertices in the dashed triangle and the solid triangle, they will pass through a point.  If we slightly perturb a vertex of either the solid or dashed triangle, the triangles will not be perspective from any point, and thus the two line arrangements will not be perspective from any line.
\end{exm}

Given a hyperplane arrangement $\A(\alpha)$ and a hyperplane of perspectivity $H_0\notin\A$, our objective is to determine all arrangements $\A'(\beta)$ perspective to $\A(\alpha)$
from $H_0$ which are weak representations of $\matroid(\alpha)$ up to rank $k$.  We accomplish this in Lemma~\ref{lem:Directions} and Proposition~\ref{prop:ParallelEquations}.

\begin{lem}\label{lem:Directions}
Suppose $\A'(\beta)$ is a hyperplane arrangement that is perspective to $\A(\alpha)$ from the hyperplane $H_0=\mathbb{V}(\alpha_0)\notin \A$.  Then we may assume $\beta_j=\alpha_j+\lambda_j\alpha_0$ is satisfied for some constant $\lambda_j\in\kk$ for $j=1,\ldots,n$.
\end{lem}
\begin{proof}
The arrangements $\A'(\beta)$ and $\A(\alpha)$ are perspective from $H_0$ if and only if $\mathbb{V}(\alpha_j, \alpha_0) = \mathbb{V}(\beta_j, \alpha_0)$ for $j=1,\ldots,n$.  It follows that $\alpha_j=\kappa_j\beta_j+\lambda_j\alpha_0$ for some scalars $\kappa_j,\lambda_j$.  Since $\alpha_j$ cannot be a multiple of $\alpha_0$, $\kappa_j\neq 0$.  Since $\mathbb{V}(\beta_j)=\mathbb{V}(\kappa_j\beta_j)$, we may replace $\beta_j$ by $\kappa_j\beta_j$.
\end{proof}

We say that the hyperplane $H_0$ is in \textit{rank $k$ general linear position} with respect to $\A(\alpha)$ if no flat in $\matroid(\alpha)$ of rank at most $k$ is  contained in $H_0$.  Lemma~\ref{lem:lingen2} shows that the property of rank $k$ general linear position for $H_0$ is passed to hyperplane arrangements perspective to $\A(\alpha)$ from $H_0$ that are weak representations of $\matroid(\alpha)$ up to rank $k$.

\begin{lem}\label{lem:lingen2}
Suppose $H_0$ is in rank $k$ general linear position with respect to a hyperplane arrangement $\A(\alpha)$, $\A'(\beta)$ is perspective to $\A$ from $H_0$, and $\beta$ is a weak representation of $\matroid(\alpha)^{[k]}$.  Then $H_0$ is also in rank $k$ general linear position with respect to $\A'$.
\end{lem}
\begin{proof}
Let $Y\in L_t(\A')$ with $t\le k$.  Without loss of generality, suppose that $\mathbb{V}(\beta_1,\ldots,\beta_t)=Y$.  Since $\beta_1, \ldots, \beta_t$ must be linearly independent, and $\beta$ is a weak representation of $\matroid(\alpha)^{[k]},$ $\alpha_1, \ldots, \alpha_t$ are also linearly independent.  Assume, for contradiction, that $Y$ is contained in $H_0$.  Then $\alpha_0$ is in the ideal $\langle \beta_1, \ldots, \beta_t\rangle,$ and $\langle \alpha_0, \alpha_1, \ldots, \alpha_t\rangle = \langle \beta_1, \ldots, \beta_t\rangle$ because $\beta_i = \alpha_i+\lambda_i \alpha_0$.  However, $\langle \alpha_0, \alpha_1,\ldots,\alpha_t\rangle$ must have codimension $t+1$ as $H_0$ is in rank $k$ general linear position with respect to $\A$, a contradiction since $\langle \beta_1,\ldots,\beta_t\rangle$ has codimension $t$.
\end{proof}

\begin{remark}
If $\beta$ is a weak representation of $\matroid(\alpha)$ up to rank $k,$ then by Lemma~\ref{lem:lingen2}, rank $k$ general linear position is preserved, but it is not guaranteed for higher ranks.  In particular, if $\{\alpha_1, \ldots, \alpha_{k+1}\}$ is a minimal dependent set satisfying $ \sum_{i=1}^{k+1} \gamma_{i}\alpha_{i}=0$ for some constants $\gamma_1,\ldots,\gamma_{k+1}$ which are not all zero, then $\{\beta_1, \ldots, \beta_{k+1}\}$ is not required to be dependent because Definition~\ref{def:weak rep} only imposes conditions on the $k$-truncation of $\matroid(\beta)$. Indeed, the relation on the $\alpha_i$ gives
\[
\sum_{i=1}^{k+1} \gamma_{i}\beta_{i}=\sum_{i=1}^{k+1} \gamma_{i}(\alpha_{i}+\lambda_{i}\alpha_0)=\left(\sum_{i=1}^{k+1}\gamma_{i}\lambda_{i}\right)\alpha_0.
 \] 
 If $\beta_1, \ldots, \beta_{k+1}$ are linearly independent then $\alpha_0 \in \langle \beta_1, \ldots, \beta_{k+1}\rangle$ and we conclude that $\cl(\beta_1, \ldots, \beta_{k+1})$ is a flat of rank $k+1$ contained in $H_0$.

\end{remark}

Not every $n$-tuple $(\lambda_j)$ determines a weak representation $\beta$ of $\matroid(\alpha)$ (or its truncations) perspective to $\A(\alpha)$ from $H_0$ as dependencies of $\matroid(\alpha)$ may not be respected.  Proposition \ref{prop:ParallelEquations} gives local conditions that imply that an arrangement perspective to $\A(\alpha)$ from a hyperplane in rank $k$ general linear position is a weak representation of $\matroid(\alpha)$ up to rank $k$.

\begin{prop}\label{prop:ParallelEquations}
Let $H_0=\mathbb{V}(\alpha_0)$ be in rank $k$ general linear position with respect to $\A(\alpha)$.
For $j=1,\ldots,n$, define $\beta_j=\alpha_j+\lambda_j\alpha_0$.  The map $\beta$ is a weak representation of $\matroid(\alpha)$ up to rank $k\ge 3$ if and only if the following condition is met for each flat $X\in L_{k-1}(\A)$:
\begin{equation}\label{eq:rels}
\mbox{if there exist } \gamma_j\in\kk \mbox{ so that } \sum_{j:X\subseteq H_j} \gamma_j\alpha_j=0, \mbox{ then } \sum_{j:X\subseteq H_j} \gamma_j\lambda_j=0.
\end{equation}
\end{prop}
\begin{proof}
First assume that $\beta$ is a weak representation of $\matroid(\alpha)$ up to rank $k$.  Suppose $X=\{\alpha_{i_1},\ldots,\alpha_{i_t}\}$ is a flat of $\matroid(\alpha)^{[k]}$ of rank $r\le k-1$.  Let $Y$ be the closure of $\{\beta_{i_1},\ldots,\beta_{i_t}\}$.  Since $k\geq 3,$ by Lemma~\ref{thm: flat map}, $Y$ is a flat of rank at most $r$ in $\matroid(\beta)$.  

Now suppose that there exist $\gamma_j\in\kk$ so that $\sum_{j:X\subseteq H_j}\gamma_j\alpha_j=0$.  We show first that $\sum_{j:X\subseteq H_j}\gamma_j\beta_j= 0$.  Suppose for the sake of contradiction that $\sum_{j:X\subseteq H_j}\gamma_j\beta_j\neq 0$.  Then
\[
\sum_{j:X\subseteq H_j}\gamma_j\beta_j=\sum_{j:X\subseteq H_j}\gamma_j(\alpha_j+\lambda_j\alpha_0)=\left(\sum_{j:X\subseteq H_j}\gamma_j\lambda_j\right)\alpha_0\neq 0.
\]
It follows that $\alpha_0\in\langle \beta_j:X\subseteq H_j\rangle$.  Thus the flat $Y$ is contained in $H_0$, contradicting Lemma~\ref{lem:lingen2}.  So $\sum_{j:X\subseteq H_j}\gamma_j\beta_j= 0$ and thus
\[
0=\sum_{j:X\subseteq H_j}\gamma_j\beta_j=\sum_{j:X\subseteq H_j}\gamma_j(\alpha_j+\lambda_j\alpha_0)=\left(\sum_{j:X\subseteq H_j}\gamma_j\lambda_j\right)\alpha_0,
\]
so $\sum_{j:X\subseteq H_j}\gamma_j\lambda_j=0$.

Now assume that Equation~\eqref{eq:rels} is satisfied for every $X\in L_{k-1}(\A)$.  Suppose that $\{\alpha_{i_1},\ldots,\alpha_{i_r}\}$ is a dependent set in $\matroid(\alpha)^{[k]}$.  If $\cl(\alpha_{i_1},\ldots,\alpha_{i_r})$ has rank $k$, then $r>k$, and any set of more than $k$ linear forms is dependent in $\matroid(\beta)^{[k]}$, so $\beta_{i_1},\ldots,\beta_{i_r}$ is dependent in $\matroid(\beta)^{[k]}$.  Now assume that $X'=\cl(\alpha_{i_1},\ldots,\alpha_{i_r})$ has rank at most $k-1$.  Since we assumed $\alpha_{i_1},\ldots,\alpha_{i_r}$ are dependent, there exist $\gamma_1,\ldots,\gamma_r\in\kk$ so that
\[
\gamma_1\alpha_{i_1}+\cdots+\gamma_r\alpha_{i_r}=0.
\]
If $\rank(X')<k-1$, take a flat $X$ of rank $k-1$ containing $X'$ (or take $X=\cap H_i$ if $k-1>\rank(\A)$). Since Equation~\eqref{eq:rels} is satisfied for $X$, we also obtain
\[
\gamma_1\beta_{i_1}+\cdots+\gamma_r\beta_{i_r}=(\gamma_1\alpha_{i_1}+\cdots+\gamma_r\alpha_{i_r})+(\gamma_1\lambda_{i_1}+\cdots+\gamma_r\lambda_{i_r})\alpha_0=0.
\]
So $\beta_{i_1},\ldots,\beta_{i_r}$ are also dependent.  It follows that $\beta$ is a weak representation of $\matroid(\alpha)$ up to rank $k$.
\end{proof}

We can reformulate the results of Proposition~\ref{prop:ParallelEquations} and Lemma~\ref{lem:Directions} as a matrix condition.
For each flat $X \in L_{k-1}(\A)$ there is an $n_X-\rank(X)$ dimensional space of linear relations among the linear forms defining $X$. Let $N^{[k]}_{\alpha}$ be the $(\sum_{X\in L_{k-1}(\A)} n_X-\rank(X)) \times n$ matrix whose rows encode the relations of~\eqref{eq:rels}.  The kernel of $N_{\alpha}^{[k]}$ is the space of tuples $(\lambda_j)_{j=1}^n\in \kk^n$ satisfying~\eqref{eq:rels}.

\begin{cor}\label{cor:NSpP}
The space of arrangements $\A'(\beta)$ which are weak representations of $\matroid(\alpha)$ up to rank $k$ and perspective to $\A(\alpha)$ from a hyperplane $H_0$ in rank $k$ general linear position with respect to $\A(\alpha)$ is isomorphic to the kernel of the matrix $N^{[k]}_{\alpha}$.
\end{cor}

\begin{notation}
\label{not:wrep}
Suppose we are given a hyperplane arrangement $\A(\alpha)\subset \mathbb{P}(V)$ with $n$ hyperplanes.  We call $\beta:\{1,\ldots,n\}\to V^*$ a \textit{(weak) \textbf{perspective}  representation} of $\matroid(\alpha)$ (or $\A(\alpha)$) \textit{up to rank} $k$, if
\begin{itemize}
\item there exists a hyperplane $H_0$ in rank $k$ general position with respect to $\A(\alpha)$ and
\item $\beta$ is a (weak) representation of $\matroid(\alpha)$ up to rank $k$ perspective from $H_0$.
\end{itemize}
Hereafter, we abbreviate \textit{weak perspective representation} as $\wrep{}$.  By Corollary~\ref{cor:NSpP}, the space of $\wrep{}$s of $\A(\alpha)$ up to rank $k$ is independent of the choice of $H_0$.  In case $k=\rank(\A)$, we simply say $\beta$ is a $\wrep{}$ of $\A(\alpha)$.  If we wish to specify the hyperplane $H_0$, we say $\beta$ is a $\wrep{}$ \textit{from }$H_0$. 
\end{notation}

In the theory of $\wrep{}$s, it is important to identify those which are \textit{trivial}.  Our intuition for trivial $\wrep{}$s is informed by the parallel drawings of Remark~\ref{rem:Parallel}.  In that case, the scalings and translations of a hyperplane arrangement are considered trivial parallel drawings.  In the following proposition we construct hyperplane arrangements which correspond to scalings and translations for arbitrary hyperplanes of perspectivity.  We will call these \textit{trivial} since every hyperplane arrangement admits them.

\begin{prop}\label{prop:defTrivialPerspectivities}
Given a hyperplane arrangement $\A(\alpha)\subset \mathbb{P}^\ell$ and a hyperplane of perspectivity $H_0$ in rank $2$ general linear position with respect to $\A(\alpha)$, define $\beta$ in either of the following ways:
\begin{enumerate}
\item for a projective linear isomorphism $T:\mathbb{P}^\ell\to\mathbb{P}^\ell$ restricting to the identity on $H_0$, set $\beta_j=\alpha_j\circ T^{-1}$.
\item for a fixed point $q \in \mathbb{P}^\ell\setminus H_0$, for every $j\neq 0$, let $\beta_j$ define the hyperplane $H'_j$ containing $q$ and the codimension two linear space $\mathbb{V}(\alpha_j,\alpha_0)$.
\end{enumerate}
In either case, $\beta$ is a $\wrep{}$ of $\A(\alpha)$ from $H_0$.
\end{prop}
\begin{proof}
For (1), since $T$ is a linear isomorphism, $\matroid(\beta) = \matroid(\alpha).$ $\A'(\beta)$ is perspective to $\A(\alpha)$ from $H_0$ because $T$ restricts to the identity on $H_0$, so $\alpha_j$ and $\beta_j$ both vanish along the codimension two linear space $H_j\cap H_0$ for all $j\neq 0$.

For (2), the assumption that $H_0$ is in rank $2$ general linear position with respect to $\A(\alpha)$ guarantees that the linear forms $\beta_j$ define distinct hyperplanes.  $\A'(\beta)$ is perspective to $\A(\alpha)$ by construction.  

Now suppose that $\alpha_1,\ldots,\alpha_s$ are dependent in $\matroid(\alpha)$.  We show that $\beta_1,\ldots,\beta_s$ are dependent in $\matroid(\beta)$.  Since $\alpha_1,\ldots,\alpha_s$ are dependent, they define a flat $X\in L_t(\A)$ of rank $t\le s-1$.  We need only show that $\beta_1,\ldots,\beta_s$ define a flat $Y$ in $\A'(\beta)$ of rank at most $s-1$.

The linear span $L$ of the chosen point $q\in \bbP^\ell\setminus H_0$ and $X\cap H_0$ has codimension $t$ (if $X\not\subset H_0$) or $t-1$ (if $X\subset H_0$, which may happen if $t>2$).  Since each of the linear forms $\beta_1,\ldots,\beta_s$ vanish at the chosen point $q$ and at $X\cap H_0$, the flat $Y$ in $\A'(\beta)$ defined by $\beta_1,\ldots,\beta_s$ contains $L$ and so has codimension at most $t\le s-1$.  Hence $\beta_1,\ldots,\beta_s$ are dependent, proving that $\matroid(\beta)$ is a weak representation of $\matroid(\alpha)$.  Note that $\rank(\A'(\beta)) = \rank(\A(\alpha))-1$ if $\A$ is essential. 
\end{proof}

The arrangements in Proposition~\ref{prop:defTrivialPerspectivities} (1) correspond to translations and scalings.  The arrangements in Proposition~\ref{prop:defTrivialPerspectivities} (2) are limits of scalings (toward $q$).

\begin{defn}
Given a hyperplane arrangement $\A(\alpha)\subset \mathbb{P}(V)$ with $n$ hyperplanes, the map $\beta:\{1,\ldots,n\}\to V^*$ is a \textit{trivial} $\wrep{}$ of $\A(\alpha)$ if there is a hyperplane $H_0$ in rank $2$ general linear position with respect to $\A(\alpha)$ so that $\beta$ can be defined in either of the ways listed in Proposition~\ref{prop:defTrivialPerspectivities}.  If we wish to emphasize the choice of hyperplane $H_0$, we say $\beta$ is a trivial $\wrep{}$ of $\A(\alpha)$ \textit{from} $H_0$.
\end{defn}

\begin{prop}\label{prop:spaceTrivialpardrawings}
A hyperplane arrangement $\A'(\beta)$ is a trivial $\wrep{}$ of $\A(\alpha)$ from $H_0=\mathbb{V}(\alpha_0)$ if and only if there is a vector $\bv\in \kk^{\ell+1}$ so that $\beta_j=\alpha_j+\alpha_j(\bv)\alpha_0$.
\end{prop}
\begin{proof}
To simplify the proof, we assume that we have chosen coordinates so that the form $\alpha_0$ defining $H_0$ is $x_0$. 

Suppose first that $q =[1:q_1:\cdots:q_\ell]\in\mathbb{P}^{\ell}\setminus H_0$,
and for each $j \neq 0$, $\beta_j$ defines the hyperplane vanishing at $q$ and also along the codimension two linear space $H_j\cap H_0$.  Let $\bv=(-1,-q_1,-q_2,\ldots,-q_\ell)$.  It is straightforward to check that $\alpha_j+\alpha_j(\bv)\alpha_0$
vanishes at $q$ and also along $H_j\cap H_0$, and so defines the same hyperplane as $\beta_j$. Therefore, since we may assume $\beta_j$ has the form $\alpha_j+\lambda_j\alpha_0$ by
 Lemma~\ref{lem:Directions}, we conclude that $\beta_j = \alpha_j+\alpha_j(\bv)\alpha_0$.

Now suppose that $\beta =\alpha\circ T^{-1}$ for some linear isomorphism $T:\bbP^\ell\to \bbP^\ell$ which restricts to the identity on $H_0=\mathbb{V}(\alpha_0)$.  Then up to scalar multiple, a matrix for $T$ must have the form
	\[
	\left[\begin{array}{c c}
	c & \mathbf{0}\\
    \mathbf{u}^T& I_\ell
	\end{array}\right],
	\]
	where $I_\ell$ is the $\ell\times\ell$ identity matrix, $\mathbf{u}=(u_1,\ldots,u_\ell)$ is a vector in $\kk^\ell$, $c$ is a nonzero constant, and $\mathbf{0}$ is the $1\times \ell$ zero vector.  The matrix for $T^{-1}$ is
	\[
	\left[\begin{array}{c c}
	\frac1c & \mathbf{0} \\
	-\frac1c \mathbf{u}^T & I_\ell
	\end{array}\right].
	\]
	Since $\beta_j=\alpha_j\circ T^{-1}$, it follows that $\beta_j=\alpha_j+\alpha_j(\bv)\alpha_0$, where \[\bv = (\frac1c-1,-u_1,-u_2,\ldots,-u_\ell).\]
	
	For the reverse direction, suppose that $\beta_j=\alpha_j+\alpha_j(\bv)\alpha_0$ for some vector $\bv=(v_0,v_1,v_2,\ldots,v_{\ell})$.  As long as $v_{0}\neq -1$, we can set  $\beta=\alpha\circ T^{-1}$, where a matrix for $T$ has the form
	\[
	\left[\begin{array}{c c c c c}
	1/(v_0+1) & 0 & 0 & \cdots & 0 \\
	-v_1/(v_0+1) & 1 & 0 & \cdots & 0 \\
	-v_2/(v_0+1) & 0 & 1 & \cdots & 0  \\
	\vdots & \vdots & \ddots & \vdots & \vdots\\
	-v_\ell/(v_0+1) & 0 & 0 & \cdots & 1 \\
	\end{array}\right],
	\]
	giving a trivial $\wrep{}$ as in Proposition \ref{prop:defTrivialPerspectivities} (1). 
	If $v_0=-1$, then $\beta_j=\alpha_j+\alpha_j(\bv)\alpha_0$ vanishes on $q=[v_0:v_1:\cdots :v_{\ell}]$ and along the codimension two linear space $H_j\cap H_0$ for all $j\neq 0$, so $\A'(\beta)$ is an arrangement of hyperplanes through $q$ and for all $j \neq 0$, $\beta_j$ defines the hyperplane through $q$ and $H_0\cap H_j$, giving a trivial $\wrep{}$ as in Proposition \ref{prop:defTrivialPerspectivities} (2).\qedhere
\end{proof}

\begin{defn}
The space of non-trivial $\wrep{}$s of $\A(\alpha)$ from $H_0$ up to rank $k$ is the space of $\wrep{}$s of $\A(\alpha)$ from $H_0$ up to rank $k$ modulo the space of trivial $\wrep{}$s of $\A(\alpha)$ from $H_0$.
\end{defn}

In what follows recall that the coefficient matrix of $\A(\alpha)$ is the $n\times (\ell+1)$ matrix $M(\alpha)$ whose $i^\text{th}$ column is given by the coefficients of $\alpha_i$.

\begin{thm}\label{thm:parameterizingparalleldrawings}
Fix a hyperplane arrangement $\A(\alpha)\subset \mathbb{P}^\ell$, an integer $1\le k\le \rank(\A)$, and a hyperplane $H_0=\mathbb{V}(\alpha_0)$ in rank $k$ general linear position with respect to $\A$.  The space of hyperplane arrangements $\A'(\beta)$ which are non-trivial  $\wrep{}$s of $\A(\alpha)$ from $H_0$ up to rank $k$ is isomorphic to the homology (at $\kk^n$) of the three-term chain complex
\begin{equation}\label{eq:perspect}
\kk^{\ell+1}\xrightarrow{M(\alpha)^T} \kk^n\xrightarrow{N^{[k]}_{\alpha}} \kk^{\sum\limits_{X\in L_{k-1}(\A)} (n_X-\rank(X))}.
\end{equation}
\end{thm}
\begin{proof}
By Corollary~\ref{cor:NSpP}, the $n$-tuple $(\lambda_j)$ defines an arrangement $\A'(\beta)$, via $\beta_j=\alpha_j+\lambda_j\alpha_0$, which is a $\wrep{}$ from $H_0$ up to rank $k$  
exactly when it is in the kernel of the matrix $N^{[k]}_\alpha$.  The image of of $M(\alpha)^T$ consists of the $n$-tuples $(\alpha_1(\bv),\alpha_2(\bv),$ $\ldots,\alpha_n(\bv))$.  By Proposition~\ref{prop:spaceTrivialpardrawings}, these tuples define \textit{trivial} $\wrep{}$s of $\matroid(\alpha)$ from $H_0$.  It follows that the homology at $\kk^n$ consists of rank $k$ $\wrep{}$s of $\A$ up to rank $k$ modulo trivial $\wrep{}$s of $\A$, hence this homology is the space of non-trivial $\wrep{}$s of $\A$ up to rank $k$.
\end{proof}

It is apparent from Theorem~\ref{thm:parameterizingparalleldrawings} that the space of non-trivial $\wrep{}$s of $\A(\alpha)$ from $H_0$ up to rank $k$ depends on the choice of $H_0$ only up to isomorphism, as long as $H_0$ is in rank $k$ general linear position.  Thus we may refer simply to the space of non-trivial $\wrep{}$s of $\A(\alpha)$ up to rank $k$.  We now draw the promised connection to formal hyperplane arrangements and more generally to $k$-generated arrangements (see Remark~\ref{rem:kgenerated} for the latter).

\begin{cor}\label{cor:FormalityAndParallelDrawings}
Let $\A(\alpha)\subset \mathbb{P}^\ell$ be a hyperplane arrangement.  Then $\A(\alpha)$ is formal if and only if $\A$ has no non-trivial $\wrep{}$s up to rank $3$.  More generally, $\A$ is $k$-generated for $k\ge 3$ if and only if $\A$ has no non-trivial  $\wrep{}$s up to rank $k$.
\end{cor}
\begin{proof}
Notice that $N_\alpha^{[3]}$ is the transpose of the matrix $\oplus_{X\in L_2(\A)} N_X$ in Corollary~\ref{cor:HomologicalCriterionFormality}.  Thus the chain complex~\eqref{eq:perspect} is the $\kk$-vector space dual of the chain complex in Corollary~\ref{cor:HomologicalCriterionFormality}.  Since taking the vector space dual is an exact functor, the homology at $\kk^n$ in the chain complex~\eqref{eq:perspect} is the dual of the homology at $\kk^n$ in~\eqref{eq:formalComplex}.  Hence one vanishes if and only if the other vanishes, which completes the proof by Corollary~\ref{cor:HomologicalCriterionFormality} and Theorem~\ref{thm:parameterizingparalleldrawings}.

The result for $k$-generated arrangements when $k>3$ follows in the same way.  We need only observe that $N_\alpha^{[k]}$ is the transpose of the matrix $\oplus_{X\in L_{k-1}(\A)} N_X$ in Remark~\ref{rem:kgenerated}.\qedhere
\end{proof}

\section{Formal line arrangements and rigid planar frameworks}\label{sec:RigidityPlanarFrameworks}

In this section we explain the connection between $\wrep{}$s of line arrangements and the rigidity of planar frameworks.  We start with a matroid $\matroid$ and explore the space of all weak representations of $\matroid$, looking for those representations $\beta:\{1,\ldots,n\}\to V^*$ where the space of non-trivial $\wrep{}$s of $\matroid(\beta)$ \textit{changes dimension}.  This is the direction we must take to explore Terao's conjecture.  In general this task is beyond the scope of this paper.  However, we scratch the surface by explaining how to explore the planar representations of matroids arising from \textit{generically minimally rigid graphs}, which is a staple of rigidity theory.

We take the line of perspectivity to be the line at infinity in $\bbP^2$ and assume that no lines in the arrangement meet each other at infinity.  In the language of~\cite[Chapter~61]{DiscCompGeom18}, a line arrangement and a $\wrep{}$ of a line arrangement from the line at infinity are both \textit{parallel drawings} of the incidence structure induced by the intersection lattice (see also Remark~\ref{rem:Parallel}).  In the plane, there is an equivalence between parallel drawings and infinitesimal rigidity -- modulo some technicalities -- which we found in an article of Whiteley~\cite{Whiteley89}.

We briefly summarize a few key points of rigidity theory and direct the reader to~\cite{SidStJohn17} or~\cite{Roth81} for a more complete introduction.
In rigidity theory, a planar bar-and joint-framework $G(\bp)$ is given by the combinatorial data of a graph $G = (V,E)$ and a placement $\bp:V \to \mathbb{R}^2$ realizing the graph in Euclidean space.  The space of \textit{infinitesimal motions} of $G(\bp)$ is the kernel of its rigidity matrix $R_{G(\bp)}$ which has one row for each edge in $G$ and two columns for each vertex.  An infinitesimal motion assigns a vector $\bv_i \in \mathbb{R}^2$ to each vertex of $G(\bp)$ such that $(\bp_i-\bp_j)\cdot (\bv_i-\bv_j) = 0$ for every edge $ij$ in $G$.

We pass from a framework to a line arrangement as follows.  Let $G$ be a graph with $n$ edges $\{e_1,\ldots,e_n\}$ and $G(\bp)$ a framework in which all edges define distinct lines.  Let $\alpha_i$ be the linear form defining the line along the $i$th bar in $G(\bp)$ ($i=1,\ldots,n$), giving an arrangement $\A_{G(\bp)} = \A(\alpha)$.  If the placement map $\bp$ is chosen generically, then the only points of $\A(\alpha)$ where at least three lines intersect are vertices of $G$. Define the generic matroid $\matroid(G)$ associated to $G$ to be $\matroid(\alpha)$ for a generic placement $\bp$.  Note that any other framework $G(\bq)$ gives rise to a weak representation of $\matroid(G)$ because maintaining the incidence structure of the graph is equivalent to maintaining dependencies among the lines along the bars.

Assuming the line at infinity is generic with respect to $\A_{G(\bp)},$ the weak representations of $\matroid(\alpha)$ that are perspective to $\A_{G(\bp)}$ from the line at infinity are precisely those weak representations consisting of lines that are parallel to the ones they correspond to in the original. We give an isomorphism between the spaces of $\wrep{}$s and infinitesimal motions via an `old engineering technique'  that Whiteley describes in \cite{Whiteley89}. 
\begin{prop}\label{prop:correspondences}
Suppose $G(\p)$ is a framework so that $\A_{G(\p)}$ has the generic matroid associated to $G$.  Then the space of infinitesimal motions of $G(\p)$ and the space of line arrangements which are $\wrep{}$s of $\A_{G(\p)}$ from the line at infinity are isomorphic.  In particular, $\A_{G(\p)}$ is formal if and only if $G(\p)$ is infinitesimally rigid.
\end{prop}
\begin{proof}
The key observation is that a $\wrep{}$ of $\A_{G(\bp)}$ from the line at infinity can be obtained by translating each vertex in $G(\bp)$ in such a way that corresponding bars are parallel.  Then we need only to check that the space of all such translations which maintain the directions of the bars is isomorphic to the space of  infinitesimal motions of $G(\bp)$.

 The space of infinitesimal motions of $G(\bp)$ is the set \[
 \{(\bv_1, \ldots, \bv_n) \in \mathbb{R}^{2n} :  (\bp_i-\bp_j)\cdot (\bv_i-\bv_j)= 0 \ \text{for every edge}\ ij \in G\}.
 \]  
 Suppose that $(\bv_1, \ldots, \bv_n) \in \mathbb{R}^{2n}$, and define $G(\bq)$ where $\bq_i = \bp_i+\bv_i^{\perp}$, where `$\perp$' represents the linear transformation of rotation by 90 degrees (consistently, either clockwise or counterclockwise).  This is the `engineering technique' of rotating the $i$th infinitesimal velocity vector by 90 degrees and then translating the $i$th vertex along it. Then
\begin{align*}
    (\bq_i-\q_j)\cdot (\bp_i-\bp_j)^{\perp} & = \big((\bp_i+\bv_i^{\perp})-(\bp_j+\bv_j^{\perp})\big)\cdot (\bp_i-\bp_j)^{\perp}\\
    & = (\bp_i-\bp_j)\cdot (\bp_i-\bp_j)^{\perp}+ (\bv_i^{\perp}-\bv_j^{\perp})\cdot (\bp_i-\bp_j)^{\perp} \\
    & = (\bv_i-\bv_j)^{\perp}\cdot (\bp_i-\bp_j)^{\perp},
\end{align*}
is zero if and only if $(\bp_i-\bp_j) \cdot (\bv_i-\bv_j) = 0$. From this computation, we see that the infinitesimal motions of $G(\bp)$ are in 1-1 correspondence with redrawings of $G(\bp)$ in which corresponding bars are parallel.  Such a redrawing is exactly a $\wrep{}$ of $\A_{G(\bp)}$ from the line at infinity, and the result follows.  The equivalence of the formality of $\A_{G(\bp)}$ and infinitesimal rigidity of $G(\bp)$ now follows from Corollary~\ref{cor:FormalityAndParallelDrawings}.
\end{proof}

With the correspondence of Proposition~\ref{prop:correspondences} in hand, we now consider the problem of searching through the (weak) representations of the generic matroid of a graph $G$.  By Proposition~\ref{prop:correspondences}, the dimension of the space of $\wrep{}$s of $\A_{G(\p)}$ is given by the rank of $R_{G(\bp)}$.  The rank of $R_{G(\p)}$ is maximized on a dense open set $U$ of realizations.  If $\bp \in U$ implies that $\rank R_{G(\bp)} = 2|V|-3$, we say that the graph $G$ is (generically) rigid, and if the deletion of any edge causes the (generic) rank to drop, we say that $G$ is generically minimally rigid.

If $G$ is generically minimally rigid, then there is a polynomial in the joint coordinates, the pure condition \cite{WhiteWhiteleyAlgGeoStresses}, which vanishes when $R_{G(p)}$ drops rank. In certain cases, the pure condition may be used in conjunction with the Grassmann-Cayley algebra to gain  geometric insight into the degenerate cases where frameworks have unexpected motions.  (See \cite{WhiteWhiteleyAlgGeoStresses}.)

There are minimally rigid graphs $G$ for which any framework $G(\bp)$ realizing the generic matroid $\matroid(G)$ has only trivial perspective matroid representations.  Of course this is already clear if $G$ is a triangle. Then $G(\bp)$ has a nontrivial infinitesimal motion if and only if the vertices of the triangle are collinear and the three lines coincide.  However, any framework on $G$ whose matroid is $\matroid(G)$ must consist of three distinct lines and therefore has no nontrivial perspective matroid representations.  In Example~\ref{ex: prism} we see this phenomenon may occur without forcing rank one flats to coincide.

\begin{exm}\label{ex: prism}
In \cite{WhiteWhiteleyAlgGeoStresses} White and Whiteley showed that the graph in Figure~\ref{fig: prism} 
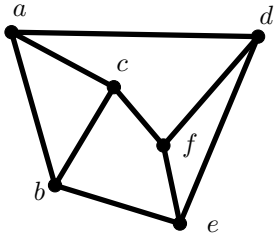
\begin{figure}
\begin{tikzpicture}[line cap=round,line join=round,>=triangle 45,x=1cm,y=1cm]
\draw [line width=2pt] (-2.5785365853658564,3.0373170731707377)-- (-2,1);
\draw [line width=2pt] (-2.5785365853658564,3.0373170731707377)-- (-1.217560975609757,2.3056097560975664);
\draw [line width=2pt] (-2,1)-- (-1.217560975609757,2.3056097560975664);
\draw [line width=2pt] (0.6995121951219533,2.978780487804884)-- (-0.33951219512195074,0.49097560975610216);
\draw [line width=2pt] (0.6995121951219533,2.978780487804884)-- (-0.5590243902439023,1.53);
\draw [line width=2pt] (-0.5590243902439023,1.53)-- (-0.33951219512195074,0.49097560975610216);
\draw [line width=2pt] (-2.5785365853658564,3.0373170731707377)-- (0.6995121951219533,2.978780487804884);
\draw [line width=2pt] (-1.217560975609757,2.3056097560975664)-- (-0.5590243902439023,1.53);
\draw [line width=2pt] (-2,1)-- (-0.33951219512195074,0.49097560975610216);
\draw [fill=black] (-2.5785365853658564,3.0373170731707377) circle (2.5pt);
\draw[color=black] (-2.4687804878048807,3.33) node {$a$};
\draw [fill=black] (-2,1) circle (2.5pt);
\draw[color=black] (-1.8834146341463434,1.2958536585365905) node [label=below left:{$b$}]{} ;
\draw [fill=black] (-1.217560975609757,2.3056097560975664) circle (2.5pt);
\draw[color=black] (-1.1078048780487812,2.598292682926835) node {$c$};
\draw [fill=black] (0.6995121951219533,2.978780487804884) circle (2.5pt) node  {};
\draw[color=black] (0.809268292682929,3.2714634146341526) node {$d$};
\draw [fill=black] (-0.33951219512195074,0.49097560975610216) circle (2.5pt);
\draw[color=black] (-0.22975609756097495,0.7836585365853705) node [label=below right:{$e$}]{};
\draw [fill=black] (-0.5590243902439023,1.53) circle (2.5pt) node [label=right:{$f$}]{};
\end{tikzpicture}
\caption{A graph $G$ for which every $\wrep{}$ with matroid $\matroid(G)$ is trivial.}
\label{fig: prism}
\end{figure}
has an infinitesimal motion if and only if either the points $a,b,c$ are collinear, the points $d,e,f$ are collinear, or if the lines $ad, be, cf$ meet at a point.  The first two conditions cause three lines to coincide.  If $G(\bp)$ is a framework satisfying the third condition, then $\A_{G(\bp)}$ has a flat of rank 2 where the three lines $ad, be,$ and $cf$ meet (in $\bbP^2$).  However, these three lines do not determine a flat of rank two in $\matroid(G)$.

Therefore, if $\A_{G(\bp)}$ has a nontrivial $\wrep{}$, its matroid cannot be equal to $\matroid(G).$
\end{exm}

Despite these cautionary examples, there are many generically rigid graphs $G$ with infinitesimally flexible embeddings possessing the generic matroid $\matroid(G)$.

\begin{exm}\label{ex: Kmn}
 A generic realization of $K_{3,3}$ in $\mathbb{R}^2$, in which the edges are fixed-length bars and the vertices are rotational joints, will be rigid.  Figure \ref{fig: K33WithVectors} shows $K_{3,3}$ realized so that one of the vertex classes is placed along the $x$-axis, and the other class is placed along the $y$-axis.  In this realization (Dixon's mechanism of the first kind), the vertices move along the axes.  Figure~\ref{fig: K33WithVectors} illustrates a perspective representation obtained via the engineer's trick described in the proof of Proposition~\ref{prop:correspondences}.  
 \begin{figure}
     \centering
     \includegraphics[width = 6cm]{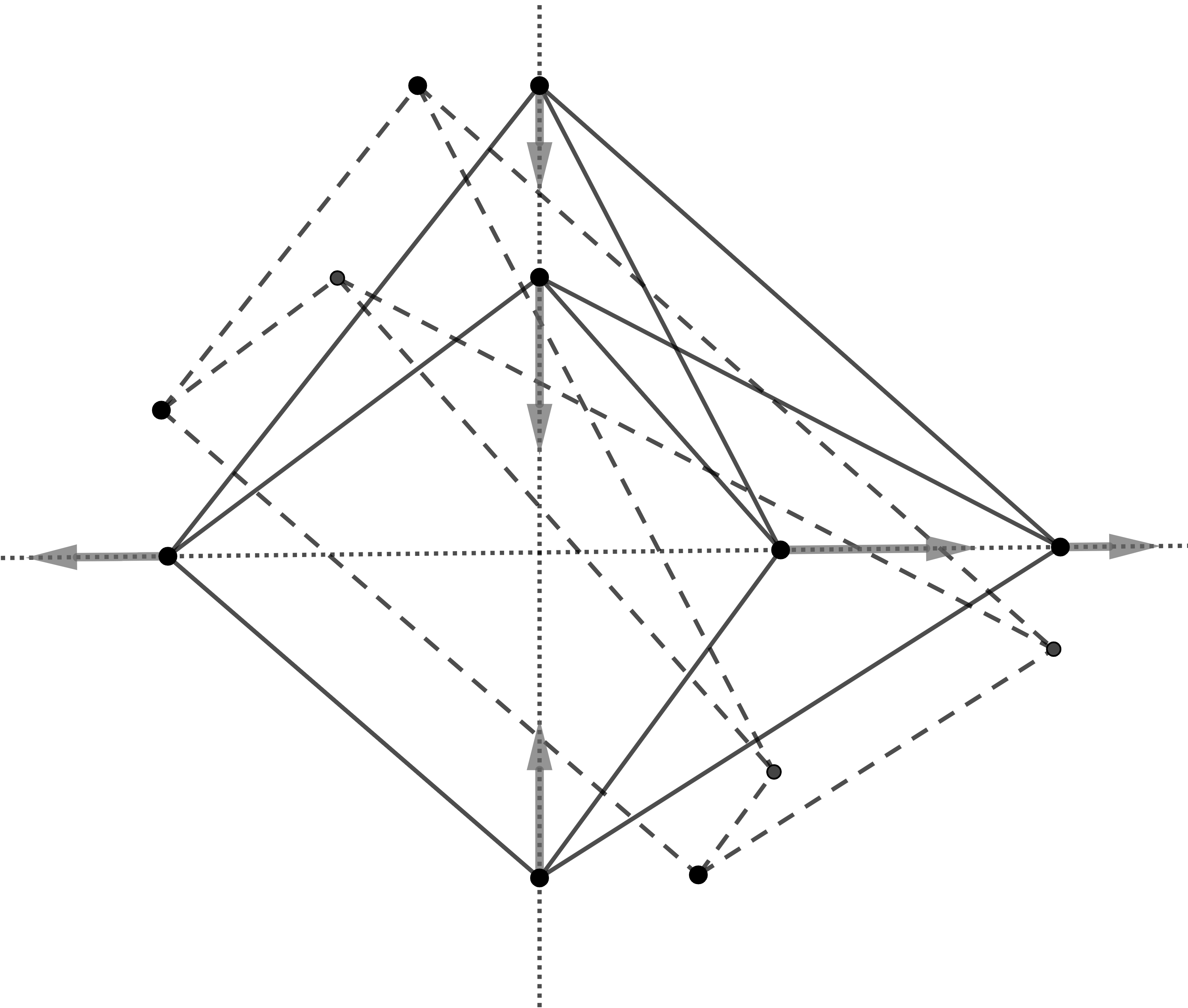}
     \caption{Dixon's mechanism of the first kind and a representation obtained via rotating velocity vectors of vertices $90^{\circ}$ clockwise.}
     \label{fig: K33WithVectors}
 \end{figure}
 This actually yields an infinite family of examples generalizing the example of Ziegler and Yuzvinsky because any $K_{s,t}$ with $s,t \geq 3$ whose two vertex classes are placed along the coordinate axes in an analogous way also has a nontrivial motion \cite{wunderlich}. Repeating the trick, we see that we have an arrangement nontrivially perspective to the first from the line at infinity.
\end{exm}

We present further examples in Section~\ref{sec:ExtremalSyzygies} after we explore the consequences of the existence of nontrivial $\wrep{}$s for syzygies of the module of logarithmic derivations.

\section{The module of logarithmic derivations and Jacobian of a hyperplane arrangement}\label{sec:ModuleDerivations}

In this section we review the constructions of $D(\A)$ and $J_{\A}$ and how their resolutions are related.  We assume, as we do throughout this work, that we are working over a field of characteristic zero.  We make connections to perspective representations in Sections \ref{sec:SaturatingJacobian} and \ref{sec:ExtremalSyzygies}.

If $\A(\alpha)=\cup_{i=1}^n H_i$ is a hyperplane arrangement in $\mathbb{P}^\ell$, then the defining polynomial of $\A$ is $Q = Q(\A):=\prod_{i=1}^n \alpha_i.$ The module of $\kk$-derivations of the polynomial ring $S=\kk[x_0,\ldots,x_\ell]$ is the free $S$-module of rank $\ell+1$, 
\[
\mbox{Der}_{\kk}(S)=\left\lbrace \sum \theta_i \frac{\partial}{\partial x_i} : \theta_i\in S \mbox{ for } i=0,\ldots,\ell\right\rbrace = \bigoplus_{i=0}^\ell S\frac{\partial}{\partial x_i}.
\]  
A derivation $\theta=\sum_{i=0}^\ell \theta_i \frac{\partial}{\partial x_i}\in\mbox{Der}_{\kk}(S)$ acts on a polynomial $F\in S$ to give the polynomial $\theta(F)=\sum_{i=0}^\ell \theta_i \frac{\partial F}{\partial x_i}$. The module of logarithmic derivations is
\[
D(\A)=\{\theta\in \mbox{Der}_{\kk}(S): \theta(Q)\in \langle Q\rangle\},
\]
where $\langle Q \rangle$ is the principal ideal of $S$ generated by $Q$.  The module of logarithmic derivations was originally introduced by Saito~\cite{Saito80} in the analytic category to study singularities of divisors, and this is an algebraic version, due to Terao~\cite{TeraoArrangementsHyperplanesFreenessI80,TeraoArrangementsHyperplanesFreenessII80}. The vector field associated to a derivation in $D(\A)$ is tangent to $\A$ at any smooth point.

The \textit{Euler derivation} $\theta_{E}=\sum x_i\frac{\partial}{\partial x_i}$ is always in $D(\A)$ because $\theta_E(Q) = (\deg Q) Q$.  In fact, as long as the characteristic of the base field does not divide the degree of $Q$ (which is equal to $n$), the sub-module generated by the Euler derivation always splits as a direct summand of  $D(\A)$.  The other summand is the rank-$\ell$ module $D_0(\A)$ consisting of those derivations $\theta\in D(\A)$ with $\theta(Q)=0$.  Since 
\[
D_0(\A)=\left\{ \theta=\sum_{i=0}^\ell \theta_i \frac{\partial}{\partial x_i}\in D(\A): \sum_{i=0}^\ell \theta_i\frac{\partial Q}{\partial x_i} =0\right\},
\]
we have an isomorphism of $D_0(\A)$ with the \textit{syzygy module} of the \textit{Jacobian ideal} $J_{\A}$ of $\A$, where
\[
J_{\A}:=\left\langle \frac{\partial Q}{\partial x_0},\ldots, \frac{\partial Q}{\partial x_\ell} \right\rangle.
\]
We write $\syz(J_{\A})$ for this syzygy module.  The isomorphism comes with a graded shift, namely
$D_0(\A)\cong \syz(J_{\A})(n-1),$
which identifies $\syz(J_{\A})$ as a sub-module of $S(-n+1)^{\ell+1}$.
The module $D_0(\A)$ is a free $S$-module exactly when $J_{\A}$ has projective dimension two, or equivalently when $J_{\A}$ is Cohen-Macaulay of codimension two.  

\begin{exm}
The $A_3$ line arrangement has defining polynomial $Q=xyz(x-y)(x-z)(y-z)$.  By Saito's criterion~\cite[Theorem~4.19]{OT92} the symmetric derivations $\psi_i = x^i\dx+y^i\dy+z^i\dz$ (where $i=1,2,3$)
generate $D(A_3)$ as an $S$-module. 
Since
\[
\begin{array}{rl}
\psi_1(Q) & =6Q\\
\psi_2(Q) & =3(x+y+z)Q\\
\psi_3(Q) & =(3x^2+3y^2+3z^2+xy+xz+yz)Q,
\end{array}
\]
$D_0(A_3) \cong \syz(J_{A_3})(5)$ is generated by $\psi_2-\frac{1}{2}(x+y+z)\psi_1$ (of degree $2$) and $\psi_3-\frac{1}{6}(3x^2+3y^2+3z^2+xy+xz+yz)\psi_1$ (of degree $3$).  These correspond to syzygies on $J_{A_3}$ of degrees 7 and 8 given by the columns of the matrix
\[
\begin{bmatrix}
x^2-\frac{1}{2}(x+y+z)x & x^3-\frac{1}{6}(3x^2+3y^2+3z^2+xy+xz+yz)x\\
y^2-\frac{1}{2}(x+y+z)y & y^3-\frac{1}{6}(3x^2+3y^2+3z^2+xy+xz+yz)y\\
z^2-\frac{1}{2}(x+y+z)z & z^3-\frac{1}{6}(3x^2+3y^2+3z^2+xy+xz+yz)z\\
\end{bmatrix}.
\]
Here, $D_0(A_3)$ is free, so its free resolution has length $1$ and this matrix gives the only map in the resolution.
\end{exm} \medskip

We relate the matroid truncation $\matroid(\alpha)^{[k]}$ for $1\le k \le \rank(\A)$ to certain saturations of the Jacobian ideal.

\begin{defn}\label{def:codimksaturation}
The codimension-$k$ saturation of the Jacobian ideal, which we denote $J^{[k]}_\A$, is the intersection of all primary components of $J_{\A}$ associated to primes of codimension (height) at most $k$.  As a saturation, we can compute this as
$
J^{[k]}_\A=J_\A:f^\infty,
$
where $f$ is in all associated primes of $J_\A$ of codimension greater than $k$ but outside all associated primes of codimension at most $k$.  The codimension $k$ saturation is a well-defined intersection of primary components of $J_\A$ --see~\cite[Proposition~3.13]{Eisenbud-1995} for a general statement regarding these types of intersections. 
\end{defn}

We prove the following formula for $J^{[k]}_{\A}$.

\begin{prop}\label{cor:SaturationHypArr}
For any hyperplane arrangement $\A$,
$
J^{[k]}_\A=\bigcap\limits_{X\in L_k(\A)} J_{\A_X}.
$
\end{prop}
We have $J^{[k]}_{\A}=J_{\A}$ whenever $k\ge \rank(\A)$. 
The codimension-$\ell$ saturation 
$J^{[\ell]}_{\A}$ is usually called \textit{the} saturation of $J_{\A}$ (that is, the saturation of $J_{\A}$ with respect to the homogeneous maximal ideal), and so we will alternatively denote $J_{\A}^{[\ell]}$ by $J_{\A}^{\sat}$.  

We set some notation and prove an auxiliary lemma before proceeding to the proof of Proposition \ref{cor:SaturationHypArr}.  If $X\in L_k(\A)$, write $I(X)$ for its ideal in $S$ and write $Q_X$ for the defining polynomial $Q(\A_X)$ of the localization $\A_X$.

\begin{lem}\label{lem:JacobianLocalizeAllDim}
For any arrangement $\A$ and any flat $X\in L(\A)$, $J_{\A_X}=S\cap(J_{\A})_{I(X)}$.  That is, $J_{\A_X}$ is the intersection of all components of $J_{\A}$ primary to an ideal contained in $I(X)$.
\end{lem}
\begin{proof}
Set $Q'=Q/Q_X=\prod_{X\not\subset H_i}\alpha_i$.  Since $Q'$ is the product of forms that do not vanish identically on $X$, each of its linear factors is invertible in $S_{I(X)}$, and hence $Q'$ is invertible in $S_{I(X)}$.  Since $Q\in J_{\A}$ by the Euler relation (this is where we need the assumption on the characteristic of the field), $Q_X=Q\cdot (1/Q')\in (J_{\A})_{I(X)}$.  For each variable $x_i\in S$ we have
\begin{equation}\label{eq:ProductRuleAllDim}
\frac{\partial Q}{\partial x_i}=\frac{\partial (Q_X\cdot Q')}{\partial x_i}=Q'\frac{\partial Q_X}{\partial x_i}+Q_X\frac{\partial Q'}{\partial x_i}.
\end{equation}
Since $Q_X\in (J_{\A})_{I(X)}$ and $Q'$ is invertible in $S_{I(X)}$, $\frac{\partial Q_X}{\partial x_i}\in S\cap(J_{\A})_{I(X)}$ for each variable $x_i$ and thus 
\[
J_{\A_X}\subset S\cap(J_{\A})_{I(X)}.
\]
Since $J_{\A}\subset J_{\A_X}$, this also shows that $(J_{\A_X})_{I(X)}=(J_{\A})_{I(X)}$.  Hence to show that
\[
J_{\A_X}= S\cap(J_{\A})_{I(X)},
\]
it suffices to prove that $J_{\A_X}=S\cap (J_{\A_X})_{I(X)}$.  Equivalently, we must show that all associated primes of $J_{\A_X}$ are contained in $I(X)$.

Suppose that the rank of $X=\cl(\alpha_1,\ldots,\alpha_k)$ is $k$.  Without loss of generality, assume that $\alpha_1,\ldots,\alpha_k$ are $x_1,\ldots,x_k$, respectively.
Then we can see that $\A_X$ is defined by a product of linear forms in the variables $x_1,\ldots,x_k$, and thus $J_{\A_X}$ is in fact an extension to $S$ of an ideal in the polynomial ring $R=\kk[x_1,\ldots,x_k]$.  Extensions of this kind are \textit{flat}; in particular if $J=I\otimes_{R} S$ for some ideal $I\subset R$, then the associated primes of $J$ are exactly the extensions to $S$ of the associated primes of $I$ (see~\cite[Theorem~23.2]{Matsumura-1986}).  It follows that any associated prime of $J_{\A_X}$ is contained in $\langle x_1,\ldots,x_k\rangle$.  Hence $J_{\A_X}=S\cap(J_{\A_X})_{I(X)}$.

The last statement follows from~\cite[Proposition~4.9]{Atiyah-Macdonald-1969}.
\end{proof}

\begin{proof}[Proof of Proposition~\ref{cor:SaturationHypArr}]
Put
\[
J'_k=\bigcap_{X\in L_k(\A)} J_{\A_X}.
\]
Suppose that $\mathcal{Q}$ is a primary component of $J_\A$ with codimension $c\le k$.  Then $\mathcal{Q}$ is primary to an ideal of the form $I(X)$, where $X\in L_c(\A)$.  By Lemma~\ref{lem:JacobianLocalizeAllDim}, $J_{\A_X}\subset \mathcal{Q}$ and hence $J'_k\subseteq \mathcal{Q}$.  Since this holds for any primary component of $J_\A$ with codimension at most $k$, $J'_k\subseteq J^{[k]}_\A$.

Now suppose $X\in L_k(\A)$.  By Lemma~\ref{lem:JacobianLocalizeAllDim}, $J_{\A_X}$ is the intersection of all primary components of $J_{\A}$ that are primary to an ideal contained in $I(X)$.  Since $J^{[k]}_{\A}$ is the intersection of all primary components of $J_{\A}$ with codimension at most $k$, we must have $J^{[k]}_{\A}\subseteq J_{\A_X}$.  Since this holds for all $X\in L_k(\A)$, we have $J^{[k]}_\A\subseteq J'_k$.
\end{proof}

\section{Perspective representations and saturating the Jacobian ideal}\label{sec:SaturatingJacobian}

Suppose $\A$ is a hyperplane arrangement with $n$ hyperplanes and defining polynomial $Q$.  The main result of this section is Theorem~\ref{thm:SaturationParallelDrawing}, which shows that polynomials in $J_{\A}^{[k]}$ of degree $n-1$ are in bijection with $\wrep{}$s of $\matroid(\alpha)$ up to rank $k+1$.  Our first step to making this connection is to compare $J_{\A}$ to the \textit{ideal of}$(n-1)$\textit{-fold products}
\[
\mathbb{I}_{n-1}(\A):=\langle Q/\alpha_1,\ldots,Q/\alpha_n\rangle.
\]
The ideals $\mathbb{I}_{n-1}(\A)$ and $J_{\A}$ are both generated in degree $n-1$ and define the singular locus of $\A$ set-theoretically. 
The ideal $\mathbb{I}_{n-1}(\A)$ is Cohen-Macaulay of codimension two -- see ~\cite[Lemma~3.1]{Garrousian-Simis-Tohaneanu-2018} for an explicit Hilbert-Burch resolution.  In fact, the primary decomposition of $\mathbb{I}_{n-1}(\A)$ is given by
\[
\mathbb{I}_{n-1}(\A)=\bigcap_{X\in L_2(\A)} I(X)^{n_X-1},
\]
see~\cite[Remark~2.5]{Garrousian-Simis-Tohaneanu-2018} and~\cite[Proposition~2.3]{Anzis-Garrousian-Tohaneanu-2017}.  From the primary decomposition above, one sees that $\mathbb{I}_{n-1}(\A)$ has exactly the same associated primes of codimension two as $J_{\A}$.  Moreover, $J_{\A_X}\subset I(X)^{n_X-1}$ for every $X\in L_2(\A)$.  Since $J_{\A_Y}\subset J_{\A_X}$ whenever $Y$ is a flat of $\A$ containing the flat $X$, we see by the description in Proposition~\ref{cor:SaturationHypArr} that $J^{[k]}_{\A}\subseteq \mathbb{I}_{n-1}(\A)$ for any $2\le k\le \rank(\A)$.

Our results hinge on writing certain polynomials of degree $n-1$ as linear combinations of the generators of $\mathbb{I}_{n-1}(\A)$.  We first observe that polynomials of degree $n-1$ in $J_{\A}$ are linear combinations of the generators of $\mathbb{I}_{n-1}(\A)$ parameterized by vectors $\bv \in \kk^{\ell+1}$. By Proposition \ref{prop:spaceTrivialpardrawings}, this gives a correspondence between elements of degree $n-1$ in $J_{\A}$ and trivial $\wrep{}$s of $\matroid(\alpha)$.

\begin{lem}\label{lem:jacobianGens}
A polynomial $f$ of degree $n-1$ is in $J_\A$ if and only if there is some vector $\bv\in\kk^{\ell+1}$ so that
\[
f=\sum_{j=1}^n \alpha_j(\bv)\frac{Q}{\alpha_j}.
\]
\end{lem}
\begin{proof}
Let $\mathbf{e}_i$ be the standard basis vector of $\kk^{\ell+1}$ that has a $1$ in position $i$ and zeros elsewhere.  We have
\[
\frac{\partial Q}{\partial x_i}=\sum_{j=1}^n \alpha_j(\mathbf{e}_i) \frac{Q}{\alpha_j}
\]
for $i=0,\ldots,\ell$.  Now extend linearly to get the result.  Explicitly, substitute this expression into the left side of the following equation and interchange summations to get:
\[
\sum_{i=0}^\ell v_i\frac{\partial Q}{\partial x_i} =\sum_{j=1}^n \alpha_j(\mathbf{v}) \frac{Q}{\alpha_j},
\]
where $v_0,\ldots,v_\ell \in\kk$ and $\mathbf{v}=(v_0,\ldots,v_\ell)$. This shows that polynomials of the given form are precisely those polynomials in $J_\A$.
\end{proof}

Since $J^{[k]}_{\A}\subset \mathbb{I}_{n-1}(\A)$, it makes sense to ask which elements of degree $n-1$ in $\mathbb{I}_{n-1}(\A)$ lie in $J_{\A}^{[k]}$.  As we see below, these elements correspond to  $\wrep{}$s of $\matroid(\alpha)$ up to rank $k$.

\begin{thm}\label{thm:SaturationParallelDrawing}
Let $\A(\alpha)\subset\bbP^\ell$ be a hyperplane arrangement with defining polynomial $\Q(\A)$, and suppose that $H_0=\mathbb{V}(\alpha_0)$ is in rank $k+1$ linearly general position with respect to $\A$.  Put
\[
f=\sum_{i=1}^n \lambda_i\frac{\Q(\A)}{\alpha_i}
\]
for some choice of $\lambda_1,\ldots,\lambda_n$.  Define $\beta:\{1,\ldots,n\}\to V^*$ by $\beta_i=\alpha_i+\lambda_i\alpha_0$.  Then $f\in J^{[k]}_{\A}$ if and only if $\A'(\beta)$ is  a  $\wrep{}$ of $\matroid(\alpha)$ from $H_0$ up to rank $k+1$. 
\end{thm}
\begin{proof}
If $X\in L_k(\A)$, define
\[
f_X=\sum_{\alpha_i\in I(X)} \lambda_i\frac{Q_X}{\alpha_i}.
\]
First we show that if $X\in L_k(\A)$, $f \in J_{\A_X}$ if and only if $f_X \in J_{\A_X}$.  Renumbering the hyperplanes of $\A$ if necessary, we assume that $\alpha_1,\ldots,\alpha_t$ vanish along $X$ and $\alpha_{t+1},\ldots,\alpha_{n}$ do not vanish identically on $X$.  Note that
\[
f=f_XQ'+\left(\sum_{i=t+1}^{n} \lambda_i\frac{Q'}{\alpha_i}\right)\cdot Q_X,
\]
where $Q_X=Q(\A_X)=\prod_{i=1}^t \alpha_i$ and $Q'=\prod_{i=t+1}^n \alpha_i$. 
Since $Q_X \in J_{\A_X}$ by the Euler relation (we are in characteristic zero), 
we see that  $f\in J_{\A_X}$ if and only if $f_XQ' \in J_{\A_X}$.  Moreover, we will show that $f_XQ'\in J_{\A_X}$ if and only if $f_X\in J_{\A_X}$.  One direction of this statement is obvious, so we explain why $f_XQ'\in J_{\A_X}$ implies that $f_X\in J_{\A_X}$.  Observe that $Q'$ is invertible in the localization $S_{I(X)}$, so if $f_XQ'\in J_{\A_X}$ then it follows that $f_X=\frac{f_XQ'}{Q'}\in (J_{\A})_{I(X)}$.  Since $f_X\in S$, we also have $f_X\in S\cap (J_{\A})_{I(X)}$.  By Lemma~\ref{lem:JacobianLocalizeAllDim}, $J_{\A_X}=S\cap (J_{\A})_{I(X)}$, so $f_X\in J_{\A_X}$.  Hence $f\in J_{\A_X}$ if and only if $f_X\in J_{\A_X}$.

Now suppose that $f\in J^{[k]}_\A$.  We prove that $\A'(\beta)$ with $\beta$ defined by $\beta_i=\alpha_i+\lambda_i\alpha_0$ is a $\wrep{}$ of $\matroid(\alpha)$ up to rank $k+1$ using Proposition \ref{prop:ParallelEquations}.
Fix $X\in L_k(\A)$.  Since $f\in J^{[k]}_\A$, it follows from Proposition~\ref{cor:SaturationHypArr} that  $f\in J_{\A_X}$.  From the discussion above, $f_X\in J_{\A_X}$.  By Lemma~\ref{lem:jacobianGens},
\[
\left(J_{\A_X}\right)_{t-1}=
\left\lbrace \sum_{i=1}^t \alpha_i(\bv)\frac{Q_X}{\alpha_i}: \bv\in \kk^{\ell+1}\right\rbrace.
\]
Hence $f_X\in J_{\A_X}$ implies there is a vector $\bv\in\kk^{\ell+1}$ so that $\lambda_i=\alpha_i(\bv)$ for $i=1,\ldots,t$. Thus whenever there are constants $\gamma_1,\ldots,\gamma_t$ so that $\sum_{i=1}^t \gamma_i\alpha_i=0$, then
\[  
\sum_{i=1}^t \gamma_i \lambda_i = \sum_{i=1}^t \gamma_i \alpha_i (\bv) =0.
\]
Since this holds for every $X \in L_k(\A)$, $\A'(\beta)$ is a $\wrep{}$ of $\matroid(\alpha)$ up to rank $k+1$ by Proposition~ \ref{prop:ParallelEquations}.

For the reverse direction, suppose that $\beta$, defined by $\beta_i=\alpha_i+\lambda_i\alpha_0$, yields a $\wrep{}$ of $\matroid(\alpha)$ up to rank $k+1$.  We show that the polynomial $f$ in the statement of the theorem satisfies $f\in J^{[k]}_\A$.  By Proposition~\ref{cor:SaturationHypArr}, it suffices to show that $f\in J_{\A_X}$ for every $X\in L_k(\A)$.  To this end, suppose $X\in L_k(\A)$ and let $\alpha_1,\ldots,\alpha_t$ be all of the linear forms of $\alpha$ which vanish along $X$.  
Since $\beta$ is a $\wrep{}$ of $\matroid(\alpha)$ up to rank $k+1$, $\beta_1,\ldots,\beta_t$ define a flat $Y$ of rank at most the rank of $X$ (which is $k$) by Lemma~\ref{thm: flat map}.  Now consider the restrictions $\alpha_X:\{1,\ldots,t\}\to V^*$ and $\beta_Y:\{1,\ldots,t\}\to V^*$.  Since $\beta$ is a $\wrep{}$ of $\matroid(\alpha)$ up to rank $k+1$, $\beta_Y$ is a $\wrep{}$ of $\matroid(\alpha_X)$ up to rank $k+1$.
By Lemma~\ref{lem:lingen2}, $Y\not\subseteq H_0$.  Hence the linear forms $\beta_1,\ldots,\beta_t$ are determined by a (any) point $q\in Y\setminus H_0$ and the codimension two linear spaces $H_0\cap H_1,\ldots,H_0\cap H_t$.  It follows that $\beta_Y$ is a trivial $\wrep{}$ of $\A_X$ of the type in Proposition~\ref{prop:defTrivialPerspectivities} (2).  By Proposition~\ref{prop:spaceTrivialpardrawings}, there is a vector $\bv\in\kk^{\ell+1}$ so that $\beta_1=\alpha_1+\alpha_1(\bv)\alpha_0,\ldots,\beta_t=\alpha_t+\alpha_t(\bv)\alpha_0$.  Hence $\lambda_i=\alpha_i(\bv)$ for each $i=1,\ldots,t$ and we conclude that $f_X \in (J_{\A_X})_{t-1}$
by Lemma~\ref{lem:jacobianGens}.  Thus $f\in J_{\A_X}$.  Since this holds for all $X\in L_k(\A),$ $f \in J^{[k]}_{\A}$ by Proposition~\ref{cor:SaturationHypArr}.
\end{proof}

\begin{cor}\label{cor:NTPerspectivitiesAndSaturationLD}
The vector space $(J^{[k]}_{\A}/J_{\A})_{n-1}$ is isomorphic to the space of non-trivial $\wrep{}$s of $\A(\alpha)$ up to rank $k+1$.
\end{cor}
\begin{proof}
By Theorem~\ref{thm:SaturationParallelDrawing}, $(J^{[k]}_{\A})_{n-1}$ is isomorphic to the space of $n$-tuples $(\lambda_i)_{i=1}^n$ which are in the kernel of the matrix $N^{[k+1]}_\alpha$ (see Corollary~\ref{cor:NSpP}).  Furthermore, by Lemma~\ref{lem:jacobianGens} and Proposition~\ref{prop:defTrivialPerspectivities}, $(J_{\A})_{n-1}$ is isomorphic to the space of $n$-tuples $(\lambda_i)_{i=1}^n$ which are the image of $M(\alpha)^T$.  The result now follows from Theorem~\ref{thm:parameterizingparalleldrawings}.
\end{proof}

\begin{cor}\label{cor:MaxPDim}
If $\A\subset\mathbb{P}^\ell$ has non-trivial $\wrep{}$s up to rank $k+1$, then $D(\A)$ has projective dimension at least $k-1$.  Equivalently, if $\A$ is not $k$-generated, then $D(\A)$ has projective dimension at least $k-1$.
\end{cor}
\begin{proof}
That the two given statements are equivalent follows from Corollary~\ref{cor:FormalityAndParallelDrawings}.  We prove the first statement.  If $\A\subset\mathbb{P}^\ell$ admits a non-trivial  $\wrep{}$ up to rank $k+1$, then $J^{[k]}_\A\neq J_\A$ by Theorem~\ref{thm:SaturationParallelDrawing}.  Hence $J_\A$ has associated primes of codimension at least $k+1$, which implies that $S/J_\A$ has projective dimension at least $k+1$.  Since $D_0(\A)$ is the second syzygy module
of $S/J_{\A}$, $D_0(\A)$ has projective dimension at least $k-1$.  Since $D_0(\A)$ is a direct summand of $D(\A)$ (as an $S$-module), the result follows. 
\end{proof}

For the next result, a hyperplane $H\in\A$ is called a \textit{separator} if $\rank(\A')<\rank(\A)$, where $\A'$ is obtained from $\A$ by removing $H$.

\begin{cor}\label{cor:GeneralPositionPdim}
Suppose $\A\subset \mathbb{P}^\ell$ is a hyperplane arrangement.  Let $H_i$ be a hyperplane of $\A$ which is not a separator, and suppose that $H_i$ is in rank $k$ general linear position with respect to the hyperplane arrangement $\A'=\A-H_i$ formed by removing $H_i$.  Then 
\[
\frac{Q(\A)}{\alpha_i}\in J^{[k]}_\A\setminus J_{\A}.
\]
In particular, the projective dimension of $D(\A)$ is at least $k-1$.
\end{cor}
\begin{proof}
Suppose $H_0$ is a hyperplane in rank $k$ linearly general position with respect to $\A$.  We claim that, for any $\lambda\neq 0\in\kk$, the map $\beta:\{1,\ldots,n\}\to V^*$ defined by
\[
\beta_j=
\left\lbrace
\begin{array}{ll}
\alpha_j+\lambda\alpha_0 & \mbox{ if } j=i \\
\alpha_j & \mbox{otherwise}
\end{array}
\right.
\]
is a non-trivial $\wrep{}$ of $\matroid(\alpha)$ up to rank $k+1$.  That $\beta$ is a $\wrep{}$ of $\matroid(\alpha)$ up to rank $k+1$ follows immediately from Lemma~\ref{lem:Directions} and Proposition~\ref{prop:ParallelEquations}, since $X\not\subset H_i$ for any flat $X\in L_k(\A)$ by assumption.  If $\beta$ is a trivial $\wrep{}$ of $\matroid(\alpha)$ then, by Proposition~\ref{prop:spaceTrivialpardrawings} there is a vector $\bv\in\kk^{\ell+1}$ so that $\beta_j=\alpha_j+\alpha_j(\bv)\alpha_0$ for all $j=1,\ldots,n$.  Thus $\alpha_j(\bv)=0$ if $j\neq i$ and $\alpha_i(\bv)=\lambda\neq 0$.  So $\bv$ is not the zero vector and thus it corresponds to a point $[\bv]\in\bbP^\ell$ so that $[\bv]\in X=\cap_{j\neq i} H_j$ and $[\bv]\notin H_i$.  It follows that $X\not\subset H_i$, and so $\rank(\A')=\codim(X)<\codim(X\cap H_i)=\rank(\A)$, contradicting our assumption that $H_i$ is not a separator.  Thus $\beta$ is non-trivial.  
It follows from Theorem~\ref{thm:SaturationParallelDrawing} and  Corollary~\ref{cor:NTPerspectivitiesAndSaturationLD} that $Q(\A)/\alpha_i\in J^{[k]}_\A\setminus J_\A$.  The statement on projective dimension follows from Corollary~\ref{cor:MaxPDim}.
\end{proof}

It is well-known that the projective dimension of $D(\A)$ is at least the projective dimension of $D(\A_X)$ for any flat $X$ (see~\cite{Yuz91,Kung-Schenck-2006}).  Thus, to conclude that the projective dimension of $D(\A)$ is at least $k-1$, it suffices to assume that there is a flat $X\subset H_i$ of rank at least $k+1$ so that the hyperplane $H_i$ is in rank $k$ general linear position with respect to $\A_X-H_i$.  Applying Corollary~\ref{cor:GeneralPositionPdim} to $\A_X$ yields that the projective dimension of $D(\A_X)$ is at least $k-1$, and thus the projective dimension of $D(\A)$ is at least $k-1$ also.

A hyperplane arrangement $\A\subset\mathbb{P}^\ell$ with at least $\ell+2$ hyperplanes is called generic if every subset of $\ell+1$ hyperplanes of $\A$ intersects in the empty set.  When $\A\subset\mathbb{P}^\ell$ is generic, it follows from Corollary~\ref{cor:GeneralPositionPdim} that the projective dimension of $D(\A)$ is $\ell-1$.  This fact has been known for some time, as it follows from more general results of Rose and Terao~\cite{RoseTeraoGeneric91} and Yuzvinsky~\cite{Yuz91}.  The local version of this -- if $\A_X$ is generic of rank $k+1$ then $D(\A)$ has projective dimension at least $k-1$ -- was observed by Kung and Schenck in~\cite{Kung-Schenck-2006}.  We also recover this using Corollary~\ref{cor:GeneralPositionPdim} and the argument in the previous paragraph.  It is worth noting that having a single hyperplane of $\A_X$ in rank $k$ general linear position is a much weaker hypothesis than assuming $\A_X$ is a generic arrangement for some $X\in L_{k+1}(\A)$.

If $\A\subset\mathbb{P}^\ell$ is generic,
then Corollary~\ref{cor:GeneralPositionPdim} also yields that the saturation of $J_{\A}$ (that is, $J^{[\ell]}_\A$) is $\mathbb{I}_{n-1}(\A)$.  This partially recovers a result of Burity, Simis, and Tohaneanu~\cite{Burity-Simis-Tohaneanu-2021}.  In fact, the latter article inspired us to compare $J_{\A}$ to $\mathbb{I}_{n-1}(\A)$.

We close this section with a proposition inspired by a result of Abe, Dimca, and Sticlaru~\cite[Proposition~4.2]{Abe-Dimca-Sticlaru-2021}.

\begin{prop}\label{prop:AlmostGeneralHyperplane}
Suppose that $\A\subset\bbP^\ell$ is a hyperplane arrangement with $n$ hyperplanes and $H$ is a hyperplane (not in $\A$) which contains a unique flat $X\in L_k(\A)$, where $k\ge 2$.  Let $\A'=\A\cup H$.  Then the space of non-trivial $\wrep{}$s of $\A$ up to rank $k+1$ is isomorphic to the space of non-trivial $\wrep{}$s of $\A'$ up to rank $k+1$.  In particular $(J_{\A}^{[k]}/J_{\A})_{n-1}\cong (J_{\A'}^{[k]}/J_{\A'})_{n}$ as vector spaces.
\end{prop}
\begin{proof}
Without loss of generality, suppose that the hyperplane $H$ is defined by the linear form $\alpha_{n+1}$ in $\A'$.  Suppose $\beta:\{1,\ldots,n\}\to V^*$ is a $\wrep{}$ of $\A$ up to rank $k+1$.  We show that $\beta$ can be lifted to a $\wrep{}$ $\beta':\{1,\ldots,n,n+1\}\to V^*$ of $\A'$ up to rank $k+1$ with $\beta'_{i}=\beta_i$ for $i=1,\ldots,n$.  With this restriction, there is only one way to define $\beta'_{n+1}$, which we now describe.  Suppose that $\alpha_1,\ldots,\alpha_t$ are the linear forms of $\alpha$ vanishing on $X$.  Then $\beta_1,\ldots,\beta_t$ define a flat $Y$ of codimension at most the codimension of $X$ by Lemma~\ref{thm: flat map}.  By Lemma~\ref{lem:lingen2}, $Y\not\subseteq H_0$.  If $\beta'$ is to be a $\wrep{}$ of $\A'$ up to rank $k+1$, then $\beta'_{n+1}$ must vanish on $Y$.  Since $Y\not\subseteq H_0$, $\beta'_{n+1}$ is determined up to constant multiple by $H\cap H_0$ and $Y$.  In fact, using the same arguments as in the proof of Theorem~\ref{thm:SaturationParallelDrawing}, there is a vector $\bv\in \kk^{\ell+1}$ so that $\beta'_i=\alpha_i+\alpha_i(\bv)\alpha_0$ for $i=1,\ldots,t,n+1$.  Thus any relations among $\alpha_1,\ldots,\alpha_t,\alpha_{n+1}$ are also relations among $\beta_1,\ldots,\beta_t,\beta'_{n+1}$.  By our assumption on $H$, the linear form $\alpha_{n+1}$ does not appear in any relations among linear forms which pass through any other flat $X'\in L_k(\A)$.  Thus by Proposition~\ref{prop:ParallelEquations}, $\beta'$ is a $\wrep{}$ of $\A'$ up to rank $k+1$ which lifts the $\wrep{}$ $\beta$ of $\A$ up to rank $k+1$.  Since every $\wrep{}$ of $\A'$ up to rank $k+1$ gives a $\wrep{}$ $\beta$ of $\A$ up to rank $k+1$ by restriction, the space of $\wrep{}$s of $\A$ up to rank $k+1$ is isomorphic to the space of $\wrep{}$s of $\A'$ up to rank $k+1$.  Under this isomorphism, trivial $\wrep{}$s of $\A$ are sent to trivial $\wrep{}$s of $\A'$, so the statement follows.

The final sentence follows from the preceding claims by Corollary~\ref{cor:NTPerspectivitiesAndSaturationLD}.
\end{proof}

\section{Extremal syzygies of line arrangements}\label{sec:ExtremalSyzygies}

In this section we exhibit an isomorphism between the space of $\wrep{}$s of a line arrangement $\A$ and the syzygies of $D_0(\A)$ of maximal degree.  We also identify non-formal arrangements as those whose module of logarithmic derivations has maximal (Castelnuovo-Mumford) regularity.  These are both consequences of a relationship between the minimal free resolution of $J^{\sat}_{\A}/J_{\A}$ and that of $D_0(\A)$, which requires more commutative algebra to develop than we have used previously.  A good reference for this material is Part 3 of~\cite{Eisenbud-1995}.

Every graded module $M$ over a polynomial ring has a graded minimal free resolution $F_{\bullet}:$
\[
0 \to F_r \xrightarrow{\partial_r} F_{r-1} \xrightarrow{\partial_{r-1}} \cdots \xrightarrow{\partial_2} F_1 \xrightarrow{\partial_1} F_0
\]
where \[F_i =  \underset{j\ge 0}{\oplus} S(-j)^{b_{i,j}},\] $\mbox{coker}(\partial_1)=M$, and all of the nonzero entries in each $\partial_i$ have degree at least one.  The number $b_{i,j}$ is the dimension of the space of minimal $i$th syzygies of $M$ of degree $j$ -- these numbers are called \textit{graded Betti numbers} of $M$.  We write $b_{i,j}(M)$ if we wish to specify the module $M$. It is standard to store the graded Betti numbers in a \textit{Betti table}
whose entry in the $j$th row and $i$th column is $b_{i,i+j}$: 

\[
\renewcommand{\arraystretch}{1}
\begin{array}{c|ccccc}
& 0 & 1 & \cdots & r-1 & r\\
\hline
0 & b_{0,0} & b_{1,1} & \cdots & b_{r-1,r-1} & b_{r,r}\\
1 & b_{0,1} & b_{1,2} & \cdots & b_{r-1,r} & b_{r, r+1}\\
\vdots & \vdots & \vdots & \ddots & \vdots & \vdots \\
m & b_{0,m} & b_{1, m+1}& \cdots & b_{r-1, r-1+m} & b_{r, r+m}
\end{array}
\]

The (Castelnuovo-Mumford) regularity of a graded $S$-module $M$ can be defined in terms of the graded Betti numbers of $M$ by
\[
\reg(M):=\max\{j-i:b_{ij}(M)\neq 0\}.
\]
In terms of the Betti table of $M$, $\reg(M)$ is the index of the last non-zero row.

\begin{exm}
Suppose that $\A$ is an arrangement of nine lines obtained by extending the edges of the realization of $K_{3,3}$ in Figure~\ref{fig:k33generic}. If the 6 vertices do not lie on a conic, then we see from the table that $D_0(\A)$ is minimally generated by 6 elements of degree 6 and that its first syzygy module is generated by 4 elements of degree 7.  The resolution has length one, and the nonzero entries of $\partial_1$ are all linear forms. The regularity of $D_0(\A)$ is 6.

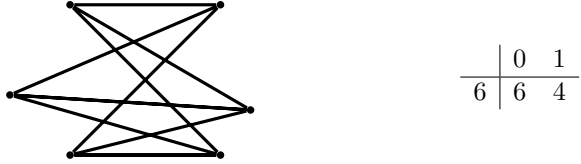
\begin{figure}[ht]
\begin{minipage}[t]{.4\textwidth}
\centering
\begin{tikzpicture}[scale=2]
\tikzstyle{dot}=[circle,fill=black,inner sep=1 pt];

\node[dot] (1) at (0,0){};
\node[dot] (2) at (-.4,-.6){};
\node[dot] (3) at (0,-1){};
\node[dot] (4) at (1,0){};
\node[dot] (5) at (1.2,-.7){};
\node[dot] (6) at (1,-1){};

\draw[very thick](1)--(4)--(2)--(5)--(1)
(1)--(6)--(3)--(4)
(2)--(5)--(3)--(6)--(2);
\end{tikzpicture}
\end{minipage}
\begin{minipage}[t]{.4\textwidth}
\centering
\raisebox{1 cm}{
$
\begin{array}{c|cc}
& 0 & 1\\
\hline
6 & 6 & 4
\end{array}
$
}
\end{minipage}
\caption{A generic realization of $K_{3,3}$ determines an arrangement $\A$; the Betti table of $D_0(\A)$ is displayed.}\label{fig:k33generic}
\end{figure}
\end{exm}

Duality properties of quotients of the form $J^{\sat}_{\A}/J_{\A}$ have been increasingly studied in the last decade -- see for example~\cite{Chardin-2004,Hassanzadeh-Simis-2012,Sernesi-2014,Straten-Warmt-2015,Failla-Flores-Peterson-2021}.  Proposition~1.3 in \cite{Hassanzadeh-Simis-2012} establishes a duality for resolutions of graded three-generated ideals of codimension two in three variables. 
For a line arrangement $\A$ with $n$ lines in $\mathbb{P}^2$, the Jacobian $J_\A$ is generated by three polynomials and this duality result shows that the resolution of $J_{\A}^{\sat}/J_{\A}$ is self-dual up to a shift (with the shifts indicated below)
\begin{equation}\label{eq:DualResolution}
 0 \rightarrow F_0^*(-3n+3) \rightarrow F_1^*(-3n+3) \rightarrow F_1 \rightarrow F_0,
\end{equation}
where $S/J_{\A} = \kk[x,y,z]/J_{\A}$ has minimal free resolution
\begin{equation}\label{eq:SJresolution}
   0 \rightarrow F_0^*(-3n+3) \rightarrow F_1^*(-3n+3) \rightarrow S^3(-n+1) \rightarrow S.
\end{equation}
The leftmost maps in Equations~\eqref{eq:DualResolution} and~\eqref{eq:SJresolution} are the same. Since $D_0(\A)$ is defined as the first syzygy module of $J_{\A}(n-1)$, we can shift the resolution of $S/J_{\A}$ in Equation~\ref{eq:SJresolution} by $n-1$ to obtain the diagram of exact sequences below:
\begin{equation}\label{eq:resolution}
\begin{tikzcd}[column sep = .5 cm, row sep= .2 cm]
0 \ar[r] & F_0^*(-2n+2) \ar[r] & F_1^*(-2n+2)\ar[rr] \ar[dr, start anchor= south east]& & S^3\ar[r] & S(n-1).\\
 & & & D_0(\A) \ar[ur]\ar[dr] &\\
 & & \phantom{HHHH} 0 \ar[ur, start anchor={[yshift=-3 pt]north east}] & & 0\\
\end{tikzcd}
\end{equation}
We can see a minimal free resolution of $D_0(\A)$ on the lefthand side of Equation~\ref{eq:resolution}. In particular, we see that the 
 generators of $D_0(\A)$ can be identified with generators of the dual of the first syzygy module of $J_{\A}^{\sat}/J_{\A}$ after a shift in degrees by $2n-2$.  Hence,
\begin{equation}\label{eq:Betti0}
b_{0,j}(D_0(\A))=b_{1,2n-2-j}(J_{\A}^{\sat}/J_{\A}),
\end{equation}
and similarly
\begin{equation}\label{eq:Betti1}
b_{1,j}(D_0(\A))=b_{0,2n-2-j}(J_{\A}^{\sat}/J_{\A}).
\end{equation}

\begin{cor}\label{cor:DerivationSyzygies}
Let $\A$ be an arrangement with $n$ lines.  Then the dimension of the vector space of minimal syzygies of degree $n-1$ of $D_0(\A)$ (that is, $b_{1, n-1}(D_0(\A))$) is equal to the dimension of the space of non-trivial $\wrep{}$s of $\matroid(\alpha)$.
\end{cor}
\begin{proof}
By Corollary~\ref{cor:NTPerspectivitiesAndSaturationLD}, the dimension of the space of non-trivial $\wrep{}$s of $\matroid(\alpha)$ is isomorphic to $(J^{\sat}_{\A}/J_{\A})_{n-1}$.  Since a basis for $(J^{\sat}_{\A}/J_{\A})_{n-1}$ consists of all minimal generators of $(J^{\sat}_{\A}/J_{\A})$ in degree $n-1$, $\dim(J^{\sat}_{\A}/J_{\A})_{n-1}=b_{0,n-1}(J^{\sat}_{\A}/J_{\A})$.  Now the result follows from Equation~\eqref{eq:Betti1}.
\end{proof}

Together, Corollary~\ref{cor:DerivationSyzygies} and Proposition~\ref{prop:correspondences} allow us to easily extract examples from the theory of rigid frameworks which generalize the phenomenon exhibited by Ziegler's pair from Section~\ref{sec:Intro}.

\begin{exm}[Examples from rigidity]\label{exm:ExtendingZiegPair}
  The first two examples can be found in~\cite[Table~1]{WhiteWhiteleyAlgGeoStresses} and the last in~\cite[Example~2]{NixonSchulzeWhiteley21}(see also~\cite[p. 88]{Whiteley78Quads}). All of the arrangements have the generic matroid for their graph (as defined in Section~\ref{sec:RigidityPlanarFrameworks}). In Figures~\ref{fig: pinchedRingSpecial} and \ref{fig: cubeWithDiagonal}, the framework is infinitesimally rigid except when the condition depicted is met, in which case the expected dimension of the space of nontrivial infinitesimal motions is one.

Figure~\ref{fig: ringOfQuads} depicts a framework that has a one-dimensional space of nontrivial infinitesimal motions for almost all embeddings.  It is easy to visually verify that the bars $\overline{a_ia_{i+1}}$ and $\overline{b_ib_{i+1}}$ are parallel and that each of these pairs has a different slope.  Since parallel lines meet on the line at infinity, the 4 intersection points are collinear, which is the condition that forces the dimension of the space of nontrivial infinitesimal motions of the framework to be two instead of one~\cite[Example~2]{NixonSchulzeWhiteley21}.  Since lines in the arrangement meet on the line at infinity, we take $H_0$ to be a different generic line.  In~\cite{NixonSchulzeWhiteley21} the authors show the same behavior carries over to any so-called \textit{ring of quadrilaterals} consisting of $n$ quadrilaterals surrounding a central $n$-gon.

We also include \textit{Betti tables} of a generic $D_0(\A_{G(\bp)})$ as well as the arrangement depicted.  In the Betti table, the number of generators of $D_0(\A_{G(\bp)})$ of degree $i$ are recorded in row $i$ and column $0$, while the number of syzygies of degree $i+1$ among the generators are recorded in row $i$ and column $1$ (add the row and column indices to get the degree of the syzygy).

\begin{figure}
\begin{floatrow}
\ffigbox[.5\textwidth]{%
 \includegraphics[width = 7cm]{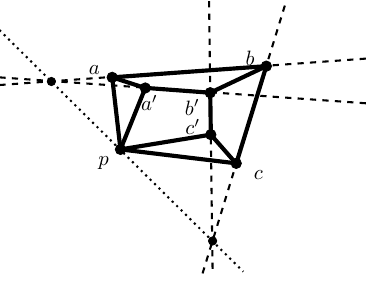}
}{%
   \caption{The points $\overline{ab} \cap \overline{a'b'}, \overline{bc} \cap \overline{b'c'},$ and $p$ are collinear.}
   \label{fig: pinchedRingSpecial}
    }
\capbtabbox[.5\textwidth]{%
\raisebox{1.4 cm}{
\begin{tabular}{c|cc}
& 0 & 1\\
\hline
7 & 1 & 0\\
8 & 6 & 5
\end{tabular}
}
\hspace{1cm}
\raisebox{1.2 cm}{
\begin{tabular}{c|cc}
& 0 & 1\\
\hline
7 & 2 & 0\\
8 & 3 & 2\\
9 & 0 & 1
\end{tabular}
}
}{%
  \caption{The generic Betti table and a special one.}%
}
\end{floatrow}
\end{figure}

\begin{figure}
\begin{floatrow}
\ffigbox{%
 \includegraphics[width = 7cm]{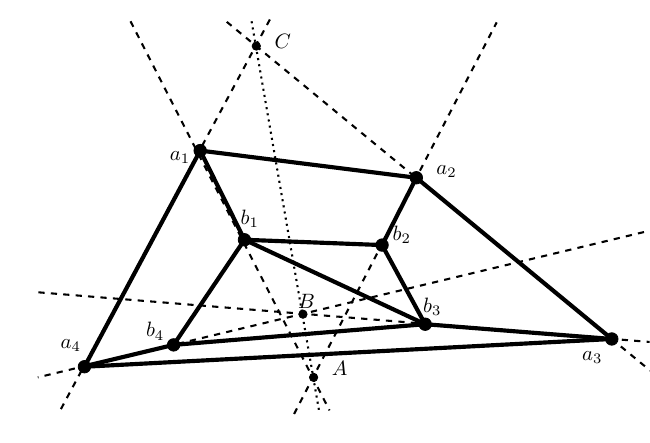}
}{%
  \caption{The points $A = \overline{a_1b_1} \cap \overline{a_2b_2}, B = \overline{a_3b_3} \cap \overline{a_4b_4},$ and $C = \overline{a_1a_4} \cap \overline{a_2a_3}$ are collinear.}
    \label{fig: cubeWithDiagonal}%
}

\capbtabbox[.5\textwidth]{%
\raisebox{1.4 cm}{
\begin{tabular}{c|cc}
     & 0 & 1\\
\hline
9 & 2 & 0\\
10 & 6 & 6
  \end{tabular}
}
\hspace{1cm}
\raisebox{1.2 cm}{
\begin{tabular}{c|cc}
& 0 & 1\\
\hline
9 & 3 & 0\\
10 & 3 & 1\\
11 & 0 & 1
\end{tabular}
}
}{%
  \caption{The generic Betti table and a special one.}%
}
\end{floatrow}
\end{figure}

 \begin{figure}
\begin{floatrow}
\ffigbox[.5\textwidth]{%
\includegraphics[width=6cm]{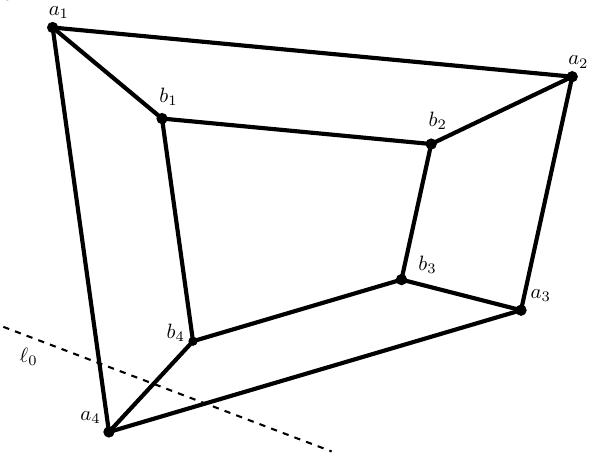}
}{
    \caption{The points $\overline{a_ia_{i+1}} \cap \overline{b_ib_{i+1}}$ lie on the line at infinity.}
     \label{fig: ringOfQuads}
}
\capbtabbox[.5\textwidth]{%
\raisebox{1.4 cm}{
\begin{tabular}{c|cc}
& 0 & 1\\
\hline
9 & 8 & 5\\
10 & 0 & 1
\end{tabular}
}
\hspace{1cm}
\raisebox{1.2 cm}{
\begin{tabular}{c|cc}
& 0 & 1\\
\hline
8 & 1 & 0\\
9 & 5 & 2\\
10 & 0 & 2
\end{tabular}
}
}{%
  \caption{The generic Betti table and a special one.}%
}
\end{floatrow}
\end{figure}
\end{exm}

Using inductive constructions from \cite{WhiteWhiteleyAlgGeoStresses}, we can glue generically minimally rigid graphs together and then place the vertices in special positions to make $b_{1,n-1}$ arbitrarily high.  For example, suppose that $G_1$ and $G_2$ are generically minimally rigid graphs.  If $uv$ is an edge in $G_1$ and $u'v'$ is an edge in $G_2$, we can create a new graph $G$ by identifying $u$ with $u'$, $v$ with $v'$ and the edge $uv$ with the edge $u'v'.$  

\begin{exm}\label{ex: glueing K33}
Form a graph $G$ by gluing together $m$ copies of $K_{3,3}$ using the aforementioned edge-gluing procedure.  For each copy of $K_{3,3}$, choose a placement map $\bp$ so that the corresponding six vertices of $G(\bp)$ lie on a conic (the conic for each copy of $K_{3,3}$ can either be the same or distinct from the conic for any other $K_{3,3}$).  Furthermore, choose $\bp$ so that $\A_{G(\bp)}$ has the generic matroid $\matroid(G)$.  Then $\A_{G(\bp)}$ will have an $m$-dimensional space of non-trivial $\wrep{}$s.  We can find an explicit basis for the space of nontrivial $\wrep{}$s with one basis element for each $K_{3,3}$ as follows:
select one $K_{3,3}$ and fix a nontrivial $\wrep{}$ of the corresponding arrangement.  This will determine a trivial $\wrep{}$ of the adjacent copies of $K_{3,3}$ via translation that propagates through all of $\A_{G(\bp)}$. Only the selected $K_{3,3}$ has been represented nontrivially.  We get one such nontrivial $\wrep{}$ for each copy of $K_{3,3}.$  By moving any of the six vertices of one of these $K_{3,3}$'s off a conic, we lose a basis element, forcing the space of non-trivial $\wrep{}$s to drop by one.

By Corollary~\ref{cor:DerivationSyzygies}, this procedure gives rise to a sequence $\A_m,\ldots,\A_1,\A_0$ of line arrangements with the same intersection lattice for which $b_{1,n-1}(D(\A_i))=i$ for $0\le i\le m$.
\end{exm}

The following corollary should be compared to~\cite[Proposition~4.14]{Abe-Dimca-Sticlaru-2021}, where Abe, Dimca, and Sticlaru study the effect on the initial degree of $D_0(\A)$ (that is, the smallest integer $r$ for which $b_{0,r}(D_0(\A))\neq 0$) when a general line is added through an intersection point of a line arrangement.  We study instead the largest degree for which $D_0(\A)$ has a syzygy.


Indeed, the following corollary shows that if $\A,\A'$ is a pair of arrangements that exhibits the behavior of Ziegler's pair, then adding a general line through an intersection point of $\A$ and a general line through an intersection point of $\A'$ will yield a pair $\mathcal{B},\mathcal{B}'$ with the same behavior (in terms of the maximum degree of a syzygy).

\begin{cor}\label{cor:adding general lines}
Suppose $\A$ and $\A'$ are line arrangements of $n$ lines with the same intersection lattice so that $b_{1,n-1}(\A)\neq b_{1,n-1}(\A')$.  Let $\mathcal{B}$ (respectively $\mathcal{B}'$) be obtained from $\A$ (respectively $\A'$) by adding a line $H$ ($H'$) which passes through a unique intersection point of $\A$ ($\A'$).  Then $b_{1,n}(\mathcal{B})=b_{1,n-1}(\A)\neq b_{1,n-1}(\A')=b_{1,n}(\mathcal{B}')$.
\end{cor}
\begin{proof}
This is immediate from Corollary~\ref{cor:DerivationSyzygies}, Corollary~\ref{cor:NTPerspectivitiesAndSaturationLD}, and Proposition~\ref{prop:AlmostGeneralHyperplane}.
\end{proof}

\begin{exm}\label{ex:addinglinetoziegler}
Consider starting with Zeigler's pair $\A$, $\A'$ where the six triple points of $\A$ are on a conic and the six triple points of $\A'$ are not on a conic.  Betti tables for $D_0(\A)$ and $D_0(\A')$ are shown in Figure~\ref{fig:addinglinetoziegler}.  We chose the particular equations $\Q(\A)=xyz(x+y+z)(2x+y+z)(2x+3y+z)(2x+3y+4z)(x+ 3z)(x+ 2y+ 3z)$ and $\Q(\A')=xyz(x+y+z)(2x+y+z)(2x+3y+z)(2x+3y+4z)(3x+5z)(3x+4y+5z)$ to compute the betti tables (the figures do not display these exact line arrangements, but just illustrate the relevant geometry).  The betti numbers $b_{1,8}(D_0(\A))=1$ and $b_{1,8}(D_0(\A'))=0$ are explained by Corollary~\ref{cor:DerivationSyzygies}.

Now add a general line through the lower left triple point of $\A$ and $\A'$ in the top of Figure~\ref{fig:addinglinetoziegler} to get the line arrangements $\B$ and $\B'$ in the bottom of Figure~\ref{fig:addinglinetoziegler}.  To compute the betti tables, we chose to add the line with equation $5x+7y+7z$ through the intersection point $[0:-1:1]$ in both $\A$ and $\A'$.  According to~\cite[Prop~4.14]{Abe-Dimca-Sticlaru-2021} (see~\cite[Example~4.16]{Abe-Dimca-Sticlaru-2021}), $b_{0,6}(D_0(\B))\neq 0$ and $b_{0,6}(D_0(\B'))\neq 0$.  According to our Corollary~\ref{cor:adding general lines}, $b_{1,9}(D_0(\A))=1$ and $b_{1,9}(D_0(\A'))=0$, as illustrated.
\end{exm}

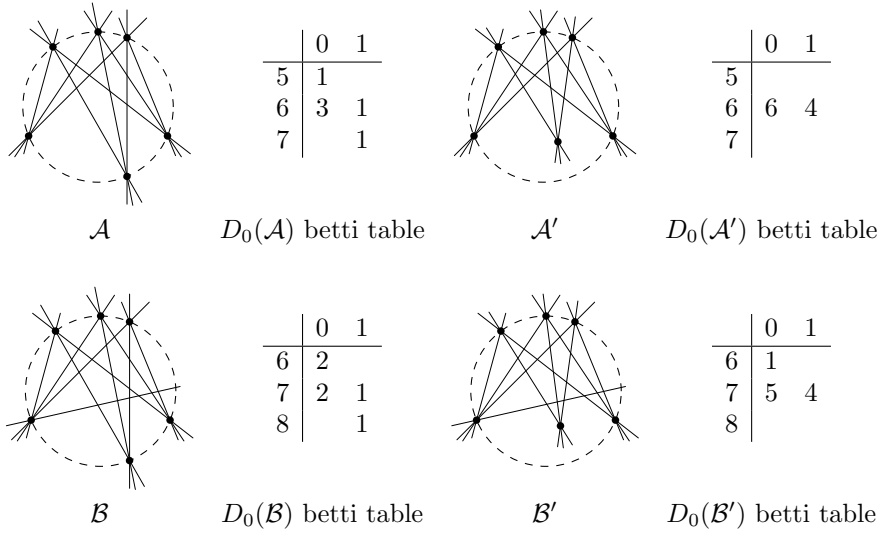
\begin{figure}

\begin{tabular}{cccc}
\begin{tikzpicture}[scale=1]
\tikzstyle{dot}=[circle,fill=black,inner sep=1 pt];

\node[dot] (0) at (-3/5,4/5){};
\node[dot] (1) at (0,1){};
\node[dot] (2) at (5/13,12/13){};
\node[dot] (3) at (12/13,-5/13){};
\node[dot] (4) at (5/13,-12/13){};
\node[dot] (5) at (-12/13,-5/13){};
				
\node at (0,-1.2){};
				
\tkzDefPoints{-.6/.8/P1, 0/1/P2, .385/.923/P3, .923/-.385/P4, .385/-.923/P5, -.923/-.385/P6}
\tkzDrawLine(P1,P4);
\tkzDrawLine(P1,P5);
\tkzDrawLine(P1,P6);
\tkzDrawLine(P2,P4);
\tkzDrawLine(P2,P5);
\tkzDrawLine(P2,P6);
\tkzDrawLine(P3,P4);
\tkzDrawLine(P3,P5);
\tkzDrawLine(P3,P6);
				
				%
				
				\draw[dashed] (0,0) circle (1);
				
			\end{tikzpicture}
			&
\raisebox{40 pt}{
\begin{tabular}{c|cc}
& 0 & 1\\
\hline
5 & 1 & \\
6 & 3 & 1\\
7 &  & 1
\end{tabular}
}
&

			\begin{tikzpicture}[scale=1]
				\tikzstyle{dot}=[circle,fill=black,inner sep=1 pt];
				
				\node[dot] (0) at (-3/5,4/5){};
				\node[dot] (1) at (0,1){};
				\node[dot] (2) at (5/13,12/13){};
				\node[dot] (3) at (12/13,-5/13){};
				\node[dot] (4) at (5/26,-12/26){};
				\node[dot] (5) at (-12/13,-5/13){};
				
				\node at (0,-1.2){};
				
				\tkzDefPoints{-.6/.8/P1, 0/1/P2, .385/.923/P3, .923/-.385/P4, .192/-.462/P5, -.923/-.385/P6}
				\tkzDrawLine(P1,P4);
				\tkzDrawLine(P1,P5);
				\tkzDrawLine(P1,P6);
				\tkzDrawLine(P2,P4);
				\tkzDrawLine(P2,P5);
				\tkzDrawLine(P2,P6);
				\tkzDrawLine(P3,P4);
				\tkzDrawLine(P3,P5);
				\tkzDrawLine(P3,P6);
				
				
				\draw[dashed] (0,0) circle (1);
			\end{tikzpicture}

&
\raisebox{40 pt}{
\begin{tabular}{c|cc}
& 0 & 1\\
\hline
5 &  & \\
6 & 6 & 4\\
7 &  & 
\end{tabular}
}
\\
$\A$ & $D_0(\A)$ betti table & $\A'$ & $D_0(\A')$ betti table\\[15 pt]

\begin{tikzpicture}[scale=1]
\tikzstyle{dot}=[circle,fill=black,inner sep=1 pt];

\node[dot] (0) at (-3/5,4/5){};
\node[dot] (1) at (0,1){};
\node[dot] (2) at (5/13,12/13){};
\node[dot] (3) at (12/13,-5/13){};
\node[dot] (4) at (5/13,-12/13){};
\node[dot] (5) at (-12/13,-5/13){};
				
\node at (0,-1.2){};
				
\tkzDefPoints{-.6/.8/P1, 0/1/P2, .385/.923/P3, .923/-.385/P4, .385/-.923/P5, -.923/-.385/P6,.727/-.012/P7}
\tkzDrawLine(P1,P4);
\tkzDrawLine(P1,P5);
\tkzDrawLine(P1,P6);
\tkzDrawLine(P2,P4);
\tkzDrawLine(P2,P5);
\tkzDrawLine(P2,P6);
\tkzDrawLine(P3,P4);
\tkzDrawLine(P3,P5);
\tkzDrawLine(P3,P6);
\tkzDrawLine(P6,P7);
				
				%
				
				\draw[dashed] (0,0) circle (1);
				
			\end{tikzpicture}

&
\raisebox{40 pt}{
\begin{tabular}{c|cc}
& 0 & 1\\
\hline
6 & 2 & \\
7 & 2 & 1\\
8 &  & 1
\end{tabular}
}
   
			&
			\begin{tikzpicture}[scale=1]
				\tikzstyle{dot}=[circle,fill=black,inner sep=1 pt];
				
				\node[dot] (0) at (-3/5,4/5){};
				\node[dot] (1) at (0,1){};
				\node[dot] (2) at (5/13,12/13){};
				\node[dot] (3) at (12/13,-5/13){};
				\node[dot] (4) at (5/26,-12/26){};
				\node[dot] (5) at (-12/13,-5/13){};
				
				\node at (0,-1.2){};
				
				\tkzDefPoints{-.6/.8/P1, 0/1/P2, .385/.923/P3, .923/-.385/P4, .192/-.462/P5, -.923/-.385/P6,.727/-.012/P7}
				\tkzDrawLine(P1,P4);
				\tkzDrawLine(P1,P5);
				\tkzDrawLine(P1,P6);
				\tkzDrawLine(P2,P4);
				\tkzDrawLine(P2,P5);
				\tkzDrawLine(P2,P6);
				\tkzDrawLine(P3,P4);
				\tkzDrawLine(P3,P5);
				\tkzDrawLine(P3,P6);
                \tkzDrawLine(P6,P7);
				
				
				\draw[dashed] (0,0) circle (1);
			\end{tikzpicture}

   &
\raisebox{40 pt}{
   \begin{tabular}{c|cc}
& 0 & 1\\
\hline
6 & 1 & \\
7 & 5 & 4\\
8 &  & 
\end{tabular}
}
\\
$\B$ & $D_0(\B)$ betti table & $\B'$ & $D_0(\B')$ betti table\\
\end{tabular}

\caption{Example~\ref{ex:addinglinetoziegler}: Illustrating Corollary~\ref{cor:adding general lines}}\label{fig:addinglinetoziegler}
\end{figure}

We next describe formal line arrangements in terms of the regularity of the module of logarithmic derivations. We first give a short proof of a result due to  Schenck~\cite[Corollary~3.5]{SchenckElementaryModifications03} using the duality of the resolution of $J^{\sat}_\A/J_{\A}$.

\begin{cor}\label{cor:Regularity}
If $\A$ is an essential arrangement with $n$ lines, then the maximal degree of a syzygy of $D_0(\A)$ is $n-1$, the maximal degree of a generator of $D_0(\A)$ is $n-2$, and thus $\reg(D_0(\A))\le n-2$.
\end{cor}
\begin{proof}
If $D_0(\A)$ is free, then $D(\A)$ is free.  In this case it is well-known that the degrees of the generators of $D(\A)$ (called the \textit{exponents} of $\A$) must add to $n$.  Since we split off the direct summand of $D(\A)$ generated by the Euler derivation (which has degree $1$) to form $D_0(\A)$, the two generators of $D_0(\A)$ must sum to $n-1$.  Since $\A$ is essential, $D_0(\A)$ does not have a generator of degree $0$.  It follows that both generators of $D_0(\A)$ have degree at most $n-2$.

Now assume that $D_0(\A)$ is not free, so $J_{\A}^{\sat}\neq J_{\A}$.
We have seen that $J_{\A}^{\sat}\subset \mathbb{I}_{n-1}(\A)$.  It follows that $J_{\A}^{\sat}$, and hence $J_{\A}^{\sat}/J_{\A}$, does not have generators in degree less than $n-1$.  In other words, $b_{0,j}(J_{\A}^{\sat}/J_{\A})=0$ for $j<n-1$.  Thus it follows from Equation~\eqref{eq:Betti1} that $b_{1,j}(D_0(\A))=0$ for $j>n-1$.  Since $J_{\A}^{\sat}/J_{\A}$ has generators in degree $n-1$, it does not have syzygies in degree less than $n$.  In other words, $b_{1,j}(J_{\A}^{\sat}/J_{\A})=0$ for $j<n$.  Thus it follows from Equation~\ref{eq:Betti0} that $b_{0,j}(D_0(\A))=0$ for $j>n-2$.  The statement about $\reg(D_0(\A))$ follows immediately.
\end{proof}

Next we identify those line arrangements $\A$ whose module of logarithmic derivations achieves the maximum possible regularity of $n-2$.  We use one more piece of terminology.  An arrangement $\A$ is \textit{irreducible} provided that it does not split as a product $\A_1\times\A_2$ of arrangements of strictly smaller dimension.  The product structure can be identified from the forms $\alpha_1,\ldots,\alpha_n$ as follows.  $\A$ splits as a product if and only if, perhaps after a suitable change of coordinates, there exist nontrivial partitions $Y_1\sqcup Y_2=\{x,y,z\}$ and $Z_1\sqcup Z_2=\{\alpha_1,\ldots,\alpha_n\}$ so that the linear forms in $Z_1$ are written solely in the variables of $Y_1$ and the linear forms in $Z_2$ are written only in the variables of $Y_2$.  For example, if $Q(\A)=xy(y-z)(y-2z)$, then $\A=\A_1\times\A_2$ where $Q(\A_1)=x$ and $Q(\A_2)=y(y-z)(y-2z)$.

\begin{cor}\label{cor:DerivationRegularity}
Suppose $\A$ is an \textit{irreducible} line arrangement with $n$ lines.  Then the following are equivalent:
\begin{enumerate}
\item $\A$ is not formal
\item $\matroid(\alpha)$ has non-trivial $\wrep{}$s
\item $\reg(D_0(\A))=n-2$
\end{enumerate}
\end{cor}
\begin{proof}
The equivalence of (1) and (2) is the content of Corollary~\ref{cor:FormalityAndParallelDrawings}.  To show (2) implies (3), assume that $\matroid(\alpha)$ has non-trivial $\wrep{}$s.  Then by Corollary~\ref{cor:DerivationSyzygies}, $b_{1,n-1}(D_0(\A))>0$.  Hence $\reg(D_0(\A))\ge n-2$.  By Corollary~\ref{cor:Regularity}, $\reg(D_0(\A))=n-2$.

Now assume (3) holds.  To prove (2) it suffices, by Corollary~\ref{cor:DerivationSyzygies}, to show that $b_{1,n-1}=b_{1,n-1}(D_0(\A))>0$.  Assume that $b_{1,n-1}=0$.  Then, since $\reg(D_0(\A))=n-2$, we must have $b_{0,n-2}=b_{0,n-2}(D_0(\A))>0$.  Thus $D_0(\A)$ must have a generator, say $\theta$, of degree $n-2$.  Since $D_0(\A)$ has no minimal syzygies of degree $n-1$ or higher (from the assumption that $b_{1,n-1}=0$), $\theta$ generates a free summand $S\theta$ of $D_0(\A)$.
Since $D_0(\A)$ has rank two, this means $D_0(\A)$ must split as a direct sum $S\theta \oplus S\psi$ for some other derivation $\psi\in D_0(\A)$.  So $\A$ is necessarily a free arrangement.  In this case it is known that the degrees of $\theta$ and $\psi$ must sum to $n-1$, so $\psi$ has degree $1$.  Thus $\A$ has \textit{exponents} $(1,1,n-2)$ (the exponents are the degrees of the free generators of $D(\A)$).  By~\cite[Proposition~4.29]{OT92}, $\A$ must split as a product of two arrangements, violating the assumption that $\A$ is irreducible.  This is a contradiction, hence $b_{1,n-1}>0$ and $\matroid(\alpha)$ has non-trivial $\wrep{}$s by Corollary~\ref{cor:DerivationSyzygies}.
\end{proof}

We give an application of Corollary~\ref{cor:DerivationRegularity} for certain line arrangements that have attracted quite a bit of attention lately.  These are the \textit{nearly free} line arrangements (nearly free curves were introduced by Dimca and Sticlaru in~\cite{Dimca-Sticlaru-Nearly-Free-2018}).  An arrangement $\A$ is called nearly free if the minimal free resolution of $D_0(\A)$ has the form
\[
0\to S(-b-1)\to S(-a)\oplus S(-b)^2\to S,
\]
where $a\le b$ (notice that $\reg(D_0(\A))=b$).  That is, $D_0(\A)$ is generated by three derivations with degrees $a,b,b$ satisfying a single relation of degree $b+1$.  This notion was generalized by Abe to \textit{plus-one generated} line arrangements in~\cite{Abe-Plus-One-2021}.  A plus-one generated line arrangement $\A$ is one whose module $D_0(\A)$ has minimal free resolution of the form
\[
0\to S(-d-1)\to S(-a)\oplus S(-b)\oplus S(-d) \to S,
\]
where $a\le b\le d$.  Abe calls $a$ and $b$ the \textit{exponents} of the plus-one generated arrangement and $d$ its \textit{level}.  Notice that $\reg(D_0(\A))=d$.  In~\cite{Abe-Plus-One-2021}, Abe showed that removing or adding a line to any free arrangement results in a plus-one generated arrangement.  It is immediate from Corollary~\ref{cor:DerivationRegularity} that if $\A$ is a plus-one generated line arrangement of $n$ lines and level $d\le n-2$, then $\A$ is formal.  It turns out that almost every nearly free line arrangement is formal.  In the following proposition, a \textit{pencil} of lines is any number of lines that pass through a single point.

\begin{prop}
Suppose $\A$ is an irreducible nearly free line arrangement.  Then $\A$ is formal unless it is obtained from a pencil of lines by adding two general lines.
\end{prop}
\begin{proof}
Since $\A$ is nearly free, $D_0(\A)$ has a minimal free resolution of the form
\[
0\to S(-b-1)\to S(-a)\oplus S(-b)^2\to S,
\]
where $a\le b$.  By~\cite[Proposition~4.1]{Abe-Plus-One-2021}, $a+b=n,$ where $n$ is the number of lines in the arrangement.  By Corollary~\ref{cor:DerivationRegularity}, $\A$ is formal if $b<n-2$, or equivalently if $a>2$.  We consider the cases where $a=1$ and $a=2$.  If $a=1$ then $\A$ is reducible by~\cite[Proposition~4.29]{OT92}, contrary to our hypothesis.  For the $a=2$ case, the line arrangements with a minimal derivation of degree two have been fully classified by Tohaneanu~\cite[Theorem~2.4]{Toh-2016}.  There are three types.  The first and third types consist of supersolvable line arrangements which are free and hence formal.  The second type consists of a pencil of lines with two general lines added to it.  These are not formal.  We claim that a pencil of lines with two general lines added is nearly free.  First of all, a pencil of lines with one general line added is free since it is a product of a one dimensional arrangement (the general line) and the pencil.  Let $\A'$ be the pencil of $k$ lines with a general line added.  Then $D_0(\A')$ must be generated in degrees $1$ and $k-1$ (since it is free, the degrees of its generators add to one less than the number of lines).  Now let $\A$ represent the line arrangement $\A'$ with another general line added.  It follows from~\cite[Theorem~1.11]{Abe-Plus-One-2021} that $D_0(\A)$ is generated in degrees $2,k,k$, with a relation in degree $k+1$.  Then $\reg(D_0(\A))=k$ and $\A$ is not formal by Corollary \ref{cor:Regularity}. Thus if $\A$ consists of two general lines added to a pencil of lines, it is nearly free but not formal.
\end{proof}

\section{Concluding Remarks and Questions}
\label{sec:questions}

In this section we make connections to the literature and close with a number of questions.  The approach we have taken in this paper provides a new and direct connection between formality and the Jacobian ideal.  We found this approach to be the most concrete in characteristic zero.  Historically, however, the connection between formality and freeness of hyperplane arrangements was established via Yuzvinsky's lattice sheaf cohomology~\cite{Yuz93A}.  M\"ucksch~\cite{Mucksch20} has recently shown that Yuzvinsky's lattice sheaf cohomology coincides with the sheaf cohomology of $\widetilde{D(\A)}$.

Brandt and Terao establish a different proof that free arrangements are formal in~\cite{Brandt-Terao-1994}, where they also introduce the notion of $k$-formality and show that a free arrangement $\A$ is $k$-formal for every $2\le k\le \rank(\A)-1$.  The construction of Brandt and Terao has recently been extended to a chain complex for multi-arrangements in~\cite{DiPasquale_2022}.  A result of Schenck and Stiller~\cite{Schenck-Stiller-2002} applied to this chain complex shows that its cohomologies are connected to the sheaf cohomology of $\widetilde{D(\A)}$ via a Cartan-Eilenberg spectral sequence.

More specifically, the following list of modules are all graded isomorphic for a line arrangement $\A$ with $n$ lines in characteristic zero:
\begin{enumerate}
\item $H^1(\cD^\bullet(\A))$, where $\cD^\bullet(\A)$ is a chain complex defined in~\cite{DiPasquale_2022} using constructions of Brandt and Terao~\cite{Brandt-Terao-1994}.
\item The lattice sheaf cohomology $H^1(L_0,\mathcal{D})$, where $L_0$ is the lattice of intersections of lines of $\A$ (disregarding the intersection of all lines) ordered by inclusion and $\mathcal{D}$ is the sheaf of derivations on $L_0$ -- studied by Yuzvinsky~\cite{Yuz91,Yuz93A,Yuz93B}
\item The sheaf cohomology module $H^1_*(\widetilde{D(\A)})=\bigoplus_{i\in\Z} H^1(\widetilde{D(\A)}(i))$
\item The local cohomology module $H^2_{\mathfrak{m}}(D(\A))$, where $\mathfrak{m}$ is the homogeneous maximal ideal.
\item $J^{\sat}_{\A}/J_{\A}(n-1)$, where $(n-1)$ represents a shift forward by $n-1$.
\end{enumerate}

The isomorphism of the last three modules is standard.  That (2) is isomorphic to $H^1_*(\widetilde{D(\A)})=\bigoplus_{i\in\Z} H^1(\widetilde{D(\A)}(i))$ has recently been shown by M\"ucksch~\cite{Mucksch20}.  That (1) is isomorphic to $H^1_*(\widetilde{D(\A)})$ follows since $D(\A)$ is a second syzygy of $H^1(\cD^\bullet)$ (we omit the details since it would take us too far afield).

In higher dimensions, modules (2)-(5) are all isomorphic again.  The isomorphism of (3),(4), and (5) is standard while the isomorphism between (2) and (3) follows from~\cite{Mucksch20}.  However, the module (1) is not necessarily isomorphic to these (as indicated above, it is connected via a Cartan-Eilenberg spectral sequence). 
We expect that many of our results -- for instance the statement in Corollary~\ref{cor:MaxPDim} that if $\A\subset\bbP^\ell$ is not $k$-generated then $D(\A)$ has projective dimension at least $k-1$ -- can be derived using either
\begin{itemize}
    \item Yuzvinsky's lattice sheaf cohomology coupled with the result of M\"uksch that Yuzvinsky's lattice sheaf cohomology coincides with the sum of twists of the sheaf cohomology of $\widetilde{D(\A)}$~\cite{Mucksch20}   or
    \item The chain complex $\cD^\bullet(\A)$ built in~\cite{DiPasquale_2022} using Brandt and Terao's constructions~\cite{Brandt-Terao-1994}.
\end{itemize}   
Either approach has potential to extend our results to fields of non-zero characteristic, although it is not clear that the results of Section~\ref{sec:ExtremalSyzygies} will extend, due to the dependence in that section on the duality inherent in the resolution of $J^\sat_\A/J_\A$.

\begin{ques}
How can the arguments of this paper be recast in terms of either Yuzvinsky's lattice sheaf cohomology (and hence the sheaf cohomology of $\widetilde{D(\A)}$ by~\cite{Mucksch20}) or the chain complex $\cD^\bullet(\A)$ from~\cite{DiPasquale_2022} which builds on constructions of Brandt and Terao~\cite{Brandt-Terao-1994}?  Which results can be extended to arbitrary characteristic?  What is the explicit connection between $\wrep{}$s of a hyperplane arrangement $\A$ and the sheaf cohomology of $\widetilde{D(\A)}$?
\end{ques}

\begin{ques}
In characteristic zero, Yuzvinsky's `first two' obstructions to freeness of arrangements in~\cite{Yuz93A} can be identified (thanks to~\cite{Mucksch20}) with the degree $n-1$ and degree $n$ parts of the $3$-saturation of $J_{\A}$ (where $n$ is the number of hyperplanes in $\A$).  In this paper we have considered only the degree $n-1$ part.

Is there a geometric interpretation for the degree $n$ part of the $k$-saturation of $J_\A$?  Can we find a rank three matroid with representations $\A(\alpha),\A'(\beta)$ so that the saturations of $J_{\A}$ and $J_{\A'}$ have the same vector space dimension in degree $n-1$ but different dimension in degree $n$?
\end{ques}

\begin{ques}
\label{ques:regularity}
If $\A\subset \bbP^\ell$ is a hyperplane arrangement with $n$ hyperplanes, Derksen and Sidman show in~\cite{Derksen-Sidman-2004} that the regularity of $D(\A)$ is bounded above by $n-1$.  If $\A$ is essential, this has recently been improved by Saito to $\reg(D(\A))\le n-\ell$, fully generalizing Schenck's bound~\cite{SchenckElementaryModifications03}.  This is sharp, since a generic arrangement $\A$ satisfies $\reg(D(\A))=n-\ell$ by~\cite{RoseTeraoGeneric91}.
Is there any connection between formality and maximal regularity when $\ell\ge 3$, as there is for line arrangements ($\ell=2$) by Corollary~\ref{cor:DerivationRegularity}?
\end{ques}

\begin{ques}\label{ques:graphs}
With a view towards Terao's conjecture, is there a generically rigid graph $G$ with generic matroid $\matroid(G)$ and two frameworks $G(\mathbf{p})$ and $G(\mathbf{q})$ so that both $\A_{G(\mathbf{p})}$ and $\A_{G(\mathbf{q})}$ have the matroid $\matroid(G)$ and
\begin{enumerate}
\item $G(\mathbf{p})$ is infinitesimally flexible (so $\A_{G(\mathbf{p})}$ is not formal and so not free) but
\item $\A_{G(\mathbf{q})}$ is free?
\end{enumerate}
\end{ques}

A positive answer to Question~\ref{ques:graphs} would provide a counterexample to Terao's conjecture.  As a first step, examine the characteristic polynomial of the matroid $\matroid(G)$.  It can often be shown that this characteristic polynomial does not factor with non-negative integer roots and hence does not admit any free representations by Terao's celebrated factorization theorem~\cite{Terao-Factorization-1981}.

\begin{ques}
The existence of non-trivial \wrep{}s for $\matroid(G)$, where $G$ is a generically minimally rigid graph, can be detected by the so-called \textit{pure condition} of White and Whiteley~\cite{WhiteWhiteleyAlgGeoStresses}.  Is there a pure condition which detects non-trivial \wrep{}s up to rank $k$ for matroids that are appropriately \textit{minimal}?
\end{ques}

Theoretically, such a `pure condition' would allow one to produce many examples in three (and higher) dimensions which exhibit behavior similar to Examples~\ref{ex: Kmn} and~\ref{exm:ExtendingZiegPair}.

\begin{ques}
In~\cite{Brandt-Terao-1994}, Brandt and Terao introduced a notion of $k$-formality which generalizes Falk and Randell's original notion.  The spaces governing $k$-formality are related to a certain graded strand of the cohomology modules $H^i_*(\widetilde{D(\A)})$ -- equivalently, the local cohomology modules of $D(\A)$ -- via a Cartan-Eilenberg spectral sequence (see~\cite{DiPasquale_2022,Schenck-Stiller-2002}).  Is there a geometric interpretation for the spaces governing $k$-formality which generalizes the notion of \wrep{}s up to rank $3$?
\end{ques}

For the next question, we ask what aspects of this paper can be extended to \textit{multi-arrangements}.  A multi-arrangement is a pair $(\A,\mathbf{m})$ consisting of a hyperplane arrangement $\A(\alpha)=\cup_{i=1}^n H_i$ and a \textit{multiplicity} function $\mathbf{m}:\{H_1,\ldots,H_n\}\to \Z_{\ge 1}$.  The module of \textit{multiderivations} $D(\A,\m)$ of $(\A,\mathbf{m})$ is defined by
\[
D(\A,\m):=\{\theta\in \mbox{Der}_{\kk}(S): \theta(\alpha_i)\in \langle \alpha_i^{\mathbf{m}(H_i)}\rangle \mbox{ for } i=1,\ldots,n \}.
\]

\begin{ques}
What (if any) of the results of this paper be extended to multi-arrangements?  Even in characteristic zero, we are not aware of a direct connection between the module of multi-derivations and the Jacobian ideal.  However, a free multi-arrangement must still be formal (and even $k$-formal)~\cite{DiPasquale_2022}.  In line with similar results from~\cite{DiPasquale_2022}, we suspect Corollary~\ref{cor:MaxPDim} carries over to multi-arrangements.  That is, if $\A$ is not $k$-generated then we suspect that $D(\A,\mathbf{m})$ has projective dimension at least $k-1$ for any multiplicity $\mathbf{m}$.
\end{ques}

Finally, Uli Walther brought to our attention the connection between the Bernstein-Sato polynomial of an arrangement and the saturation of its Jacobian ideal, which is worked out in~\cite[Section~5.3]{Walther-2017}.  Since we have shown in Theorem~\ref{thm:SaturationParallelDrawing} that formality is closely tied to polynomials of maximal degree in the saturation of the Jacobian ideal, we raise the following question.

\begin{ques}\label{ques:BernsteinSato}
Can explicit connections be made between formality of an arrangement and its Bernstein-Sato polynomial?
\end{ques}

\section{Acknowledgments}
We thank Dan Bath for pointing out to us Saito's improvement~\cite[Proposition~1.3]{saito2019degeneration} on the bound from~\cite{Derksen-Sidman-2004}.  Uli Walther directed us to the connection of Ziegler's pair with the Bernstein-Sato polynomial, and also to the work of Ruud Pellikaan on the duality of $J^{\sat}_{\A}/J_{\A}$.  We also thank Aron Simis, Hal Schenck, Walter Whiteley, Stefan Tohaneanu, and Max Wakefield for their comments on earlier drafts.  The first author acknowledges partial support from NSF grant DMS-2201084.




\end{document}